\documentclass[11pt,a4paper]{article}
\usepackage{amssymb}
\usepackage{amsmath}
\usepackage{amsthm}
\usepackage{latexsym}
\usepackage{dsfont}
\newtheorem{thm}{Theorem}
\newtheorem{lem}[thm]{Lemma}
\newtheorem{cor}[thm]{Corollary}

\usepackage{xpatch}
\xpatchcmd{\proof}{\itshape}{\normalfont\proofnameformat}{}{}
\newcommand{\proofnameformat}{}

\begin{document}
	
	\renewcommand{\proofnameformat}{\bfseries}
	
	\begin{center}
		{\Large\textbf{Limit laws for cotangent and Diophantine sums}}
		
		\vspace{10mm}
		
		\textbf{Bence Borda, Lorenz Fr\"uhwirth and Manuel Hauke}
		
		\vspace{5mm}
		
		{\footnotesize Graz University of Technology
			
			Steyrergasse 30, 8010 Graz, Austria
			
			Email: \texttt{borda@math.tugraz.at, fruehwirth@math.tugraz.at, hauke@math.tugraz.at}}
		
		\vspace{5mm}
		
		
	\end{center}
	
	\vspace{5mm}
	
	\begin{abstract} Limit laws for ergodic averages with a power singularity over circle rotations were first proved by Sinai and Ulcigrai, as well as Dolgopyat and Fayad. In this paper, we prove limit laws with an estimate for the rate of convergence for the sum $\sum_{n=1}^N f(n \alpha)/n^p$ in terms of a $1$-periodic function $f$ with a power singularity of order $p \ge 1$ at integers. Our results apply in particular to cotangent sums related to Dedekind sums, and to sums of reciprocals of fractional parts, which appear in multiplicative Diophantine approximation. The main tools are Schmidt's method in metric Diophantine approximation, the Gauss--Kuzmin problem and the theory of $\psi$-mixing random variables.
	\end{abstract}
	
	\section{Introduction}
	
	Cotangent and cosecant sums such as
	\[ \sum_{n=1}^N \frac{\cot (\pi n \alpha)}{n}, \qquad \sum_{n=1}^N \frac{1}{n \sin (\pi n \alpha)}, \qquad \sum_{n=1}^N \frac{1}{n^p |\sin (\pi n \alpha)|^q} \]
	at various irrational $\alpha$ were first studied by Hardy and Littlewood \cite{HL1,HL2,HL3} in connection with Diophantine approximation and lattice point counting. Further connections to Dedekind sums, real quadratic fields and modularity relations were explored in \cite{BE,BER,BS,CM1,CM2,LRR,RI}. The sums are sensitive to the Diophantine properties of $\alpha$. For instance, the first two sums are $O(\log N)$ for a badly approximable $\alpha$, and this is sharp in general. In this paper, we consider similar sums at a randomly chosen $\alpha$, and find their limit distributions.
	
	Throughout, $\alpha$ is a $[0,1]$-valued random variable whose distribution $\mu$ is absolutely continuous with respect to the Lebesgue measure $\lambda$, denoted by $\mu \ll \lambda$. Sometimes we also assume that $\mu$ has a positive Lipschitz density, that is,
	\begin{equation}\label{lipschitz}
		\left| \frac{\mathrm{d} \mu}{\mathrm{d}\lambda} (x) - \frac{\mathrm{d} \mu}{\mathrm{d}\lambda} (y) \right| \le L |x-y| \qquad \textrm{and} \qquad \frac{\mathrm{d} \mu}{\mathrm{d}\lambda} (x) \ge A
	\end{equation}
	for all $x,y \in [0,1]$ with some constants $L \ge 0$ and $A>0$. Throughout the paper, $L$ and $A$ always denote the constants in \eqref{lipschitz}.
	
	Let $\mathrm{Stab}(\alpha_0, \beta_0)$ denote the standard stable law with stability index $\alpha_0 \in (0,2]$ and skewness parameter $\beta_0 \in [-1,1]$, defined as the probability distribution on $\mathbb{R}$ whose characteristic function is
	\[ \varphi (t)= \exp \left( - |t|^{\alpha_0} \left( 1- i \beta_0 \, \mathrm{sgn} (t) \omega \right) \right) \quad \textrm{with} \quad \omega = \left\{ \begin{array}{ll} \tan \frac{\alpha_0 \pi}{2} & \textrm{if } \alpha_0 \neq 1, \\ - \frac{2}{\pi} \log |t| & \textrm{if } \alpha_0=1 . \end{array} \right. \]
	We write simply ``Cauchy'' for the standard Cauchy distribution $\mathrm{Stab}(1,0)$. We use the notation $\overset{d}{\to}$ for convergence in distribution, and $X \sim \nu$ if a random variable $X$ has distribution $\nu$. If $X_N \overset{d}{\to} \nu$ and $\nu$ has a continuous cumulative distribution function $F$, then the rate of convergence in the Kolmogorov metric is defined as $\sup_{x \in \mathbb{R}} |\Pr (X_N \le x) - F(x)|$.
	
	\begin{thm}\label{cotangenttheorem} Let $\alpha \sim \mu$ with some $\mu \ll \lambda$. Then
		\[ \frac{1}{\log N} \sum_{n=1}^N \frac{\cot (\pi n \alpha)}{n} \overset{d}{\to} \mathrm{Cauchy} \qquad \textrm{and} \qquad \frac{1}{\log N} \left( \sum_{n=1}^N \frac{|\cot (\pi n \alpha)|}{n} -E_N \right) \overset{d}{\to} \mathrm{Stab}(1,1), \]
		where $E_N = \frac{1}{\pi} (\log N)^2 + \frac{2}{\pi} \log N \log \log N -c \log N$ with the constant
		\[ c= \frac{2}{\pi} \left( \gamma + \log \frac{\pi^2}{6} - \frac{12 \zeta'(2)}{\pi^2} -1 \right) = 0.77338986 \ldots . \]
		Under the assumption \eqref{lipschitz}, in both relations the rate of convergence in the Kolmogorov metric is $\ll (\log N)^{-1/4} (\log \log N)^{1/4}$ with an implied constant depending only on $L$ and $A$.
	\end{thm}
	
	Here $\gamma$ denotes the Euler--Mascheroni constant and $\zeta'(2)=-\sum_{n=1}^{\infty} n^{-2} \log n$ is the derivative of the Riemann zeta function at the point 2. We also prove similar limit laws for cotangent sums with higher powers. For instance, we will show that if $\alpha \sim \mu$ with some $\mu \ll \lambda$, then
	\begin{equation}\label{cotangentsquare}
		\frac{1}{\frac{4}{5\pi} (\log N)^2} \sum_{n=1}^N \frac{\cot^2 (\pi n \alpha)}{n^2} \overset{d}{\to} \mathrm{Stab}(1/2,1) ,
	\end{equation}
	as well as
	\begin{equation}\label{cotangentcube}
		\frac{1}{\sigma (\log N)^3} \sum_{n=1}^N \frac{\cot^3 (\pi n \alpha)}{n^3} \overset{d}{\to} \mathrm{Stab} (1/3,0) \quad \textrm{and} \quad \frac{1}{\sigma (\log N)^3} \sum_{n=1}^N \frac{|\cot (\pi n \alpha)|^3}{n^3} \overset{d}{\to} \mathrm{Stab} (1/3,1)
	\end{equation}
	with $\sigma = 64/(35 \Gamma (1/3)^3)$, where $\Gamma$ is the gamma function. Under the assumption \eqref{lipschitz}, the rate of convergence in the Kolmogorov metric is $\ll (\log N)^{-1/3} (\log \log N)^{1/3}$ in \eqref{cotangentsquare} resp.\ $\ll (\log N)^{-1/3}$ in both relations in \eqref{cotangentcube}.
	
	Theorem \ref{cotangenttheorem} and relations \eqref{cotangentsquare} and \eqref{cotangentcube} are illustrations of limit laws for the more general sum $\sum_{n=1}^N f(n \alpha) /n^p$ in terms of a $1$-periodic function $f$ with a $1/x^p$ or $1/|x|^p$-type singularity at integers, see the main results Theorems \ref{maintheorem} and \ref{p>1theorem} and Corollaries \ref{maincorollary} and \ref{p>1corollary} below.
	
	Our results should be compared to a theorem of Sinai and Ulcigrai \cite{SU}, who proved that if $(\alpha, \beta)$ is a random variable uniformly distributed on the unit square, then the normalized sum $N^{-1} \sum_{n=1}^N \cot (\pi (n \alpha + \beta))$ converges in distribution. The limit distribution was not explicitly identified, however. Their proof is based on the renewal theory of continued fractions, and applies to more general sums $N^{-1} \sum_{n=1}^N f(n \alpha + \beta)$ in terms of a $1$-periodic function with a $1/x$-type singularity at integers. Higher order power singularities were investigated by Dolgopyat and Fayad \cite{DF} using dynamical methods, in which case the limit distributions can be described in terms of random lattices. See also Kesten \cite{KE1,KE2} for the case of bounded $f$. The proofs in \cite{DF,SU} do not yield a rate of convergence.
	
	In contrast, we obtain classical stable laws as limit distributions, and an estimate for the rate of convergence. Our proofs are based on a combination of three main tools. Firstly, we generalize a classical method of Schmidt in metric Diophantine approximation \cite{HA,SCH}. Secondly, we use quantitative mixing properties of continued fractions related to the Gauss--Kuzmin problem. Finally, we use the theory of $\psi$-mixing random variables from probability theory \cite{BR}.
	
	\section{Main results}\label{mainsection}
	
	Let $\langle \cdot \rangle$ resp.\ $\| \cdot \|$ denote the signed resp.\ unsigned distance from the nearest integer function. That is, let $\langle x \rangle = x$ and $\| x \| = |x|$ on the interval $x \in (-1/2,1/2]$, and extend them with period $1$.
	
	Let $\nu_1 \otimes \nu_2$ denote the product of the measures $\nu_1$ and $\nu_2$. If a sequence of $\mathbb{R}^2$-valued random variables $X_N$ converges in distribution to $\nu$ and $\nu$ has a continuous cumulative distribution function $F$, then the rate of convergence in the Kolmogorov metric is defined similarly as
	\[ \sup_{(x,y) \in \mathbb{R}^2} | \Pr (X_N \in (-\infty, x] \times (-\infty, y] ) - F(x,y) | . \]
	
	\subsection{Singularities of order 1}
	
	Our main result is a joint limit law for the sum of positive and negative parts of $f(n\alpha)/n$, where $f$ has a $1/x$-type singularity at integers. The indicator of a set or relation $S$ is denoted by $\mathds{1}_S$.
	\begin{thm}\label{maintheorem} Let $f_1(x)=\frac{1}{\langle x \rangle} \mathds{1}_{\{ \langle x \rangle >0 \}} +g_1(x)$ and $f_2(x)=\frac{1}{|\langle x \rangle|} \mathds{1}_{\{ \langle x \rangle <0 \}} +g_2(x)$ with some $1$-periodic functions $g_1$ and $g_2$ that are of bounded variation on $[0,1]$. Let $\alpha \sim \mu$ with some $\mu \ll \lambda$. Then
		\[ \left( \frac{\sum_{n=1}^N f_1(n \alpha) /n - E_{1,N}}{(\pi /2) \log N}, \frac{\sum_{n=1}^N f_2(n \alpha) /n - E_{2,N}}{(\pi /2) \log N} \right) \overset{d}{\to} \mathrm{Stab} (1,1) \otimes \mathrm{Stab} (1,1) , \]
		where $E_{j,N} = \frac{1}{2} (\log N)^2 + \log N \log \log N - c_j \log N$, $j=1,2$, with the constants
		\[ c_j= \gamma + \log \frac{2 \pi}{3} - \frac{12 \zeta'(2)}{\pi^2} -1 - \int_0^1 g_j(x) \, \mathrm{d}x . \]
		Assuming \eqref{lipschitz}, the rate of convergence in the Kolmogorov metric is $\ll (\log N)^{-1/4} (\log \log N)^{1/4}$ with an implied constant depending only on $L$, $A$, $g_1$ and $g_2$.
	\end{thm}
	
	If $(X,Y) \sim \mathrm{Stab}(1,1) \otimes \mathrm{Stab}(1,1)$, then $(X-Y)/2 \sim \mathrm{Cauchy}$ and $(X+Y)/2-(2 \log 2)/\pi \sim \mathrm{Stab}(1,1)$, as seen from the characteristic functions. The limit law of the difference resp.\ sum of the coordinates of the random vector in Theorem \ref{maintheorem} thus corresponds to the following.
	\begin{cor}\label{maincorollary} Let $g$ be a $1$-periodic function that is of bounded variation on $[0,1]$. Let $\alpha \sim \mu$ with some $\mu \ll \lambda$. Then
		\[ \frac{1}{\pi \log N} \left( \sum_{n=1}^N \frac{\frac{1}{\langle n \alpha \rangle} + g(n \alpha)}{n} - \log N \int_0^1 g(x) \, \mathrm{d}x \right) \overset{d}{\to} \mathrm{Cauchy} \]
		and
		\[ \frac{1}{\pi \log N} \left( \sum_{n=1}^N \frac{\frac{1}{\| n \alpha \|} + g (n \alpha)}{n} - E_N \right) \overset{d}{\to} \mathrm{Stab}(1,1) , \]
		where $E_N = (\log N)^2 + 2 \log N \log \log N - c \log N$ with the constant
		\[ c=2 \left( \gamma + \log \frac{\pi}{3} - \frac{12 \zeta'(2)}{\pi^2} -1 \right) - \int_0^1 g(x) \, \mathrm{d}x . \]
		Under the assumption \eqref{lipschitz}, in both relations the rate of convergence in the Kolmogorov metric is $\ll (\log N)^{-1/4} (\log \log N)^{1/4}$ with an implied constant depending only on $L$, $A$ and $g$.
	\end{cor}
	
	Theorem \ref{cotangenttheorem} is a special case of Corollary \ref{maincorollary}. Indeed, $\pi \cot (\pi x) = 1/\langle x \rangle + g(x)$ with a $1$-periodic function $g$ of bounded variation on $[0,1]$ such that $\int_0^1 g(x) \, \mathrm{d}x =0$. Similarly, $\pi |\cot (\pi x)| =1/\| x \| + g(x)$ with a $1$-periodic function $g$ of bounded variation on $[0,1]$ such that $\int_0^1 g(x) \, \mathrm{d}x = 2 \log (2/\pi)$.
	
	Corollary \ref{maincorollary} with $g=0$ yields limit laws for the Diophantine sums
	\[ \sum_{n=1}^N \frac{1}{n \langle n \alpha \rangle} \qquad \textrm{and} \qquad \sum_{n=1}^N \frac{1}{n \| n \alpha \|}. \]
	As a further example, note that Theorem \ref{maintheorem} applies to the functions $f_1(x)=1/\{ x \}$ and $f_2(x)=1/(1-\{ x \})$, where $\{ \cdot \}$ denotes the fractional part function, with some $g_1$, $g_2$ such that $\int_0^1 g_1 (x) \, \mathrm{d}x = \int_0^1 g_2(x) \, \mathrm{d}x = \log 2$, yielding a joint limit law for the Diophantine sums
	\[ \left( \sum_{n=1}^N \frac{1}{n \{ n \alpha \}}, \sum_{n=1}^N \frac{1}{n (1-\{ n \alpha \})} \right) . \]
	The almost sure asymptotics of these sums was found in \cite{BO2,ER}, see also \cite{KR}. Higher-dimensional analogues were studied in \cite{FR1,FR2,LV} using Fourier analytic and dynamical methods, but no limit law is known in this case. Diophantine sums of this form play an important role in multiplicative Diophantine approximation \cite{BHV}.
	
	Diophantine counting problems with shrinking targets, such as
	\[ \left| \left\{ 1 \le n \le N \, : \, \| n \alpha \| < \delta_n \right \} \right| \]
	are closely related, and satisfy the central limit theorem as well as various other probabilistic limit theorems under suitable assumptions on the sequence $\delta_n \to 0$, see Philipp \cite{PH}, Samur \cite{SA2}, Fuchs \cite{FU1,FU2,FU3} and references therein. Central limit theorems for higher dimensional Diophantine counting problems were proved by Dolgopyat, Fayad and Vinogradov \cite{DFV} and Bj\"orklund and Gorodnik \cite{BG1,BG2,BG3} using dynamical methods.
	
	\subsection{Higher order singularities}
	
	We also prove a limit law for functions with a higher order singularity at integers.
	\begin{thm}\label{p>1theorem} Let $f_1(x)=\frac{1}{\langle x \rangle^p} \mathds{1}_{\{ \langle x \rangle >0 \}} +g_1(x)$ and $f_2(x)=\frac{1}{|\langle x \rangle|^p} \mathds{1}_{\{ \langle x \rangle <0 \}} +g_2(x)$ with some real constant $p>1$ and some $1$-periodic functions $g_1$ and $g_2$ that satisfy $|g_1(x)|, |g_2(x)| \ll 1/\| x \|^{p'}$ with some real constant $0 \le p' <p$. Let $\alpha \sim \mu$ with some $\mu \ll \lambda$. Then
		\[ \left( \frac{\sum_{n=1}^N f_1(n \alpha) / n^p +c \log N}{\sigma (\log N)^p}, \frac{\sum_{n=1}^N f_2(n \alpha) / n^p +c \log N}{\sigma (\log N)^p} \right) \overset{d}{\to} \mathrm{Stab} (1/p,1) \otimes \mathrm{Stab}(1/p,1), \]
		where $c=\mathds{1}_{\{ 1<p<2 \}} c_p$ with some explicit constant $c_p$, and $\sigma= \zeta(2p) \left( \frac{6}{\pi^2} \cos \left( \frac{\pi}{2p} \right) \Gamma \left( 1-\frac{1}{p} \right) \right)^p$. Assuming \eqref{lipschitz}, the rate of convergence in the Kolmogorov metric is
		\[ \ll \left\{ \begin{array}{ll} (\log N)^{-\frac{p}{2(p+1)}}(\log \log N)^{\frac{p}{2(p+1)}} & \textrm{if } 1<p \le 1+\sqrt{3}, \\ (\log N)^{-1/p} & \textrm{if } p> 1+\sqrt{3} \end{array} \right. \]
		with an implied constant depending only on $L$, $A$, $g_1$, $g_2$ and $p$.
	\end{thm}
	The centering term $c \log N$ is not necessary for the limit law itself, only for the rate. In fact, it is only visible in the range $1<p<(1+\sqrt{17})/4 =1.28\ldots$ (when $1-p>-\frac{p}{2(p+1)}$). The explicit formula for the constant $c_p$, $1<p<2$ is
	\begin{equation}\label{cp}
		c_p = \frac{6 \zeta (2p)}{\pi^2} \left( \frac{1}{p+1} + \frac{p-1}{2(p+1)} \log \frac{4}{3} +B_p \right),
	\end{equation}
	where
	\[ \begin{split} B_p &= \int_1^{\infty} \frac{x^p - \lfloor x \rfloor^p + x^{p-1}}{x^2+x} \, \mathrm{d}x - \frac{1}{p+1} \int_1^{\infty} \frac{x^{p-2} (6x+4)}{(x+1)(x+2)} \, \mathrm{d}x \\ &\phantom{={}}+ \int_1^{\infty} \left( \{ x \} -\frac{1}{2} \right) x^{p-1} \left( p \log \left( 1+\frac{1}{x(x+2)} \right) - \frac{2}{(x+1)(x+2)} \right) \, \mathrm{d}x . \end{split} \]
	The limit law of the difference resp.\ sum of the random vector in Theorem \ref{p>1theorem} corresponds to the following.
	\begin{cor}\label{p>1corollary} Let $p>1$ be a real constant, and let $g$ be a $1$-periodic function that satisfies $|g(x)| \ll 1/\| x \|^{p'}$ with some real constant $0 \le p' <p$. Let $\alpha \sim \mu$ with some $\mu \ll \lambda$. Then
		\[ \frac{1}{2^p \sigma (\log N)^p} \sum_{n=1}^N \frac{\frac{\mathrm{sgn} (\langle n \alpha \rangle)}{|\langle n \alpha \rangle|^p}+g(n \alpha)}{n^p} \overset{d}{\to} \mathrm{Stab}(1/p,0) \]
		and
		\[ \frac{1}{2^p \sigma (\log N)^p} \left( \sum_{n=1}^N \frac{\frac{1}{\| n \alpha \|^p}+g(n \alpha)}{n^p} +2c \log N \right) \overset{d}{\to} \mathrm{Stab}(1/p,1) , \]
		where $c$ and $\sigma$ are as in Theorem \ref{p>1theorem}. Assuming \eqref{lipschitz}, in both relations the rate of convergence in the Kolmogorov metric is the same as in Theorem \ref{p>1theorem} with an implied constant depending only on $L$, $A$, $g$ and $p$.
	\end{cor}
	
	Limit laws for cotangent sums such as relations \eqref{cotangentsquare} and \eqref{cotangentcube} immediately follow as a special case.
	
	The rest of the paper is organized as follows. We recall the mixing properties of continued fractions related to the Gauss--Kuzmin problem in Section \ref{mixingsection}. The proof of Theorem \ref{maintheorem} and Corollary \ref{maincorollary} is given in Section \ref{proofmaintheoremsection}, with an overview of the proof strategy in Section \ref{overviewsection}. The proof of Theorem \ref{p>1theorem} and Corollary \ref{p>1corollary} is given in Section \ref{p>1section}.
	
	\section{Mixing properties of continued fractions}\label{mixingsection}
	
	We refer to Khinchin \cite{KH} for a general introduction to Diophantine approximation and continued fractions, to Iosifescu and Kraaikamp \cite{IK} for a comprehensive account of the stochastic properties of continued fractions, and to Bradley \cite{BR} for a survey of various mixing coefficients and their applications in probability theory.
	
	We work on the probability space $([0,1], \mathcal{B}, \mu)$, where $\mathcal{B}$ is the family of Borel sets and $\mu \ll \lambda$, and write $\mathbb{E}_{\mu}$, $\mathrm{Var}_{\mu}$ and $\mathrm{Cov}_{\mu}$ for the expected value, the variance and the covariance. The continued fraction expansion of $\alpha \in [0,1]$ is written as $\alpha = [0;a_1,a_2,\ldots]$ with convergents $p_k/q_k=[0;a_1,a_2,\ldots, a_k]$. The partial quotients $a_k$ are thus random variables, and we use the common shorthand notation, say, $\mu (a_k=3) = \mu (\{ \alpha \in [0,1] \, : \, a_k=3 \})$ for the probability of the event $a_k=3$. The most important special case is the Gauss measure $\mu_{\mathrm{Gauss}}(B)=\frac{1}{\log 2} \int_B \frac{1}{1+x} \, \mathrm{d}x$, $B \in \mathcal{B}$.
	
	Let $\mathcal{A}_{\ell}^k$ be the $\sigma$-algebra generated by $a_j$, $\ell \le j \le k$, and similarly let $\mathcal{A}_{\ell}^{\infty}$ be the $\sigma$-algebra generated by $a_j$, $j \ge \ell$. A classical form of the Gauss--Kuzmin theorem \cite[p.\ 97]{IK} states that under the assumption \eqref{lipschitz},
	\begin{equation}\label{gausskuzmin}
		\left| \mu (B) - \mu_{\mathrm{Gauss}}(B)  \right| \ll e^{-a\ell} \mu_{\mathrm{Gauss}}(B) \quad \textrm{uniformly in } B \in \mathcal{A}_{\ell}^{\infty}
	\end{equation}
	with a small constant $a>0$ and an implied constant depending only on $L$. Under the assumption \eqref{lipschitz}, the $\psi$-mixing coefficients of the sequence of partial quotients $a_1, a_2, \ldots$, defined as
	\[ \psi (\ell) = \sup_{k \ge 1} \sup_{\substack{B \in \mathcal{A}_1^k, \,\, \mu (B)>0 \\ C \in \mathcal{A}_{k+\ell}^{\infty}, \,\, \mu (C)>0}} \left| \frac{\mu (B \cap C)}{\mu (B) \mu (C)} -1 \right|, \qquad \ell \ge 1, \]
	satisfy $\psi (\ell) \ll e^{-a \ell}$, where the constant $a>0$ and the implied constant depend only on $L$ and $A$. In the special case $\alpha \sim \mu_{\mathrm{Gauss}}$ the sequence $a_1, a_2, \ldots$ is strictly stationary with distribution
	\[ \mu_{\mathrm{Gauss}} (a_k = n) = \frac{1}{\log 2} \log \left( 1+\frac{1}{n (n+2)} \right), \qquad k,n \ge 1. \]
	
	The mixing properties of the partial quotients lead to various probabilistic limit theorems for the sum $\sum_{k=1}^K f(a_k)$ depending on the growth rate of the function $f$. In particular, for any $\mu \ll \lambda$,
	\begin{equation}\label{sumaklimitlaw}
		\frac{\sum_{k=1}^K a_k - A_K}{\frac{\pi}{2 \log 2} K} \overset{d}{\to} \mathrm{Stab} (1,1),
	\end{equation}
	where $A_K = \frac{1}{\log 2} K \log K + \frac{1}{\log 2} \left( \log \frac{\pi}{2 \log 2} - \gamma \right) K$. Heinrich \cite{HE} proved that in the special case $\alpha \sim \mu_{\mathrm{Gauss}}$ the rate of convergence in the Kolmogorov metric is $\ll (\log K)^2 /K$. See Lemma \ref{sumakplimitlawlemma} in Section \ref{p>1section} for a similar limit law with rate for power sums $\sum_{k=1}^K a_k^p$, $p>1$.
	
	\section{Proof of Theorem \ref{maintheorem}}\label{proofmaintheoremsection}
	
	\subsection{Overview}\label{overviewsection}
	
	The proof of Theorem \ref{maintheorem} consists of three main steps. We are interested in the Diophantine sums
	\[ \sum_{\substack{n=1 \\ \langle n \alpha \rangle >0}}^N \frac{1}{n \langle n \alpha \rangle} \qquad \textrm{and} \qquad \sum_{\substack{n=1 \\ \langle n \alpha \rangle <0}}^N \frac{1}{n |\langle n \alpha \rangle|} . \]
	We treat the terms such that $\langle n \alpha \rangle \ge 1/(2n)$ resp.\ $\langle n \alpha \rangle \le -1/(2n)$, and the terms such that $0<\langle n \alpha \rangle < 1/(2n)$ resp.\ $-1/(2n) < \langle n \alpha \rangle<0$ separately. The threshold is chosen as $1/(2n)$ so that in the latter case $n$ is an integral multiple of some convergent denominator $q_k$ by the theorem of Legendre \cite[p.\ 30]{KH}.
	
	By following the approach of Schmidt \cite{HA,SCH} with certain modifications, in Section \ref{awaysection} we show that the sum of all terms that are away from the singularity concentrates around $\frac{1}{2} (\log N)^2$:
	\[ \begin{split} \sum_{\substack{n=1 \\ \langle n \alpha \rangle \ge 1/(2n)}}^N \frac{1}{n \langle n \alpha \rangle} &= \frac{1}{2} (\log N)^2 + o(\log N) \quad \textrm{in measure} , \\ \sum_{\substack{n=1 \\ \langle n \alpha \rangle \le - 1/(2n)}}^N \frac{1}{n |\langle n \alpha \rangle|} &= \frac{1}{2} (\log N)^2 + o(\log N) \quad \textrm{in measure} . \end{split} \]
	In Section \ref{closesection} we estimate the sum of the terms that are close to the singularity. Very roughly, we will show that
	\[ \sum_{\substack{n=1 \\ 0<\langle n \alpha \rangle < 1/(2n)}}^N \frac{1}{n \langle n \alpha \rangle} \approx \frac{\pi^2}{6} \sum_{\substack{1 \le k \le \frac{12 \log 2}{\pi^2} \log N \\ k \textrm{ odd}}} a_k, \qquad \sum_{\substack{n=1 \\ -1/(2n) < \langle n \alpha \rangle <0}}^N \frac{1}{n |\langle n \alpha \rangle|} \approx \frac{\pi^2}{6} \sum_{\substack{1 \le k \le \frac{12 \log 2}{\pi^2} \log N \\ k \textrm{ even}}} a_k . \]
	The appearance of the sum of partial quotients explains the limit distribution $\mathrm{Stab}(1,1)$ in Theorem \ref{maintheorem} and Corollary \ref{maincorollary}. As we will see, the sum of partial quotients with odd resp.\ even indices are asymptotically independent, which leads to a product measure as limit distribution in Theorem \ref{maintheorem}.
	
	Finally, in Section \ref{completesection} we find the contribution of a $1$-periodic function of bounded variation $g$ by showing that
	\[ \sum_{n=1}^N \frac{g(n \alpha)}{n} = \log N \int_0^1 g(x) \, \mathrm{d}x + o(\log N) \quad \textrm{in measure} . \]
	Throughout the proof, we give quantitative error terms under the assumption \eqref{lipschitz}.
	
	\subsection{Reduction to positive Lipschitz densities}
	
	The following lemma shows that it will be enough to prove all our results under the assumption \eqref{lipschitz}.
	\begin{lem}\label{lipschitzlemma} Let $X_N: [0,1] \to \mathbb{R}^n$ be a sequence of Borel measurable functions, and let $\nu$ be a Borel probability measure on $\mathbb{R}^n$. Assume that $X_N(\alpha) \overset{d}{\to} \nu$ whenever $\alpha \sim \mu$ with some $\mu \ll \lambda$ that satisfies \eqref{lipschitz}. Then $X_N(\alpha) \overset{d}{\to} \nu$ whenever $\alpha \sim \mu$ with some $\mu \ll \lambda$.
	\end{lem}
	
	\begin{proof} Let $\mu \ll \lambda$ be arbitrary with density function $h=\frac{\mathrm{d}\mu}{\mathrm{d}\lambda}$. For any $\varepsilon>0$ there exists a smooth, positive function $h_{\varepsilon}$ on $[0,1]$ such that $\int_0^1 h_{\varepsilon} (x) \, \mathrm{d}x=1$ and $\int_0^1 |h(x)-h_{\varepsilon}(x)| \, \mathrm{d} x < \varepsilon$. Let $\mu_{\varepsilon}$ be the Borel probability measure on $[0,1]$ with density $h_{\varepsilon}$. In particular, $|\mu (B) - \mu_{\varepsilon}(B)|<\varepsilon$ for any $B \in \mathcal{B}$.
		
		Let $S \subseteq \mathbb{R}^n$ be any Borel set such that $\nu (\partial S)=0$. Since $\mu_{\varepsilon}$ has a positive Lipschitz density, by assumption $\mu_{\varepsilon} (X_N \in S) \to \nu (S)$. As $\varepsilon >0$ was arbitrary, it follows that $\mu (X_N \in S) \to \nu (S)$. Hence $X_N(\alpha) \overset{d}{\to} \nu$ for $\alpha \sim \mu$, as claimed.
	\end{proof}
	
	For the rest of Section \ref{proofmaintheoremsection}, let $\alpha \sim \mu$ with some $\mu \ll \lambda$ that satisfies \eqref{lipschitz}. A small constant $a>0$ whose value changes from line to line, and all implied constants depend only on $L$ and $A$.
	
	\subsection{Away from the singularity}\label{awaysection}
	
	Let $f_n: \mathbb{R} \to \mathbb{R}$ be a sequence of functions that vanish on $(-\infty, 0)$, are nonnegative, nonincreasing, and integrable on $[0,\infty)$. This is a more general setup than the one used by Schmidt \cite{SCH} and Harman \cite{HA}, who considered indicators of intervals $f_n = \mathds{1}_{[0,\delta_n]}$ under suitable assumptions on the sequence $\delta_n \ge 0$. Let
	\[ I_n = \int_0^{\infty} f_n (x) \, \mathrm{d}x, \quad F(N)= \sum_{n=1}^N I_n, \quad \gamma_n (x) = \sum_{j=0}^{n-1} f_n (nx - j), \quad \gamma_{n,k}(x) = \sum_{\substack{j=0 \\ \mathrm{gcd}(j,n) \le k}}^{n-1} f_n (nx - j) . \]
	
	\begin{lem}\label{schmidtlemma} Let $1 \le k \le N$ be integers, and assume that there exists a constant $C>0$ such that
		\begin{equation}\label{Ind}
			\sum_{\substack{n=1 \\ d \mid n}}^N I_n \le \frac{C}{d} \sum_{n=1}^N I_n
		\end{equation}
		for all positive integers $d$. Then
		\[ \int_0^{\infty} \left( \sum_{n=1}^N \gamma_n (x) - \sum_{n=1}^N \gamma_{n,k} (x) \right) \, \mathrm{d}x \le C \frac{F(N)}{k}, \]
		and
		\[ \int_0^{\infty} \left( \sum_{n=1}^N \gamma_{n,k} (x) \right)^2 \mathrm{d}x - \left( \int_0^{\infty} \sum_{n=1}^N \gamma_{n,k}(x) \, \mathrm{d}x \right)^2 \le \sum_{n,m=1}^N A (k,n,m) \int_0^{\infty} f_n (nx) f_m (mx) \, \mathrm{d}x + 2C \frac{F(N)^2}{k} \]
		with some nonnegative integers $A(k,n,m)=A(k,m,n)$ that satisfy $\sum_{n=1}^{\ell} A(k,n,m) \le \ell d_k(m)$ for any positive integers $\ell$ and $m$, where $d_k(m) = |\{ 1 \le d \le k \, : \, d \mid m \}|$.
	\end{lem}
	
	\begin{proof} Let $\varphi (k,n) = |\{ 0 \le j \le n-1 \, : \, \mathrm{gcd} (j,n) \le k \}|$. We have $\int_0^{\infty} f_n (nx -j) \, \mathrm{d}x = I_n/n$ for all $0 \le j \le n-1$, hence
		\[ \int_0^{\infty} \left( \gamma_n (x) - \gamma_{n,k} (x) \right) \, \mathrm{d}x = I_n \left( 1-\frac{\varphi (k,n)}{n} \right) . \]
		Summing over $1 \le n \le N$ and using the estimate
		\[ 1-\frac{\varphi (k,n)}{n} = \frac{1}{n} \sum_{\substack{d \mid n \\ d>k}} |\{ 0 \le j \le n-1 \, : \, \mathrm{gcd} (j,n)=d \}| \le \sum_{\substack{d \mid n \\ d>k}} \frac{1}{d}, \]
		assumption \eqref{Ind} leads to
		\[ \int_0^{\infty} \left( \sum_{n=1}^N \gamma_n (x) - \sum_{n=1}^N \gamma_{n,k} (x) \right) \, \mathrm{d}x \le \sum_{n=1}^N I_n \sum_{\substack{d \mid n \\ d>k}} \frac{1}{d} = \sum_{d=k+1}^N \frac{1}{d} \sum_{\substack{n=1 \\ d \mid n}}^N I_n \le \sum_{d=k+1}^N C \frac{F(N)}{d^2} \le C \frac{F(N)}{k}, \]
		as claimed.
		
		To prove the second claim of the lemma, observe that
		\[ \left( \int_0^{\infty} \sum_{n=1}^N \gamma_{n,k}(x) \, \mathrm{d}x \right)^2 = \sum_{n,m=1}^N I_n I_m \frac{\varphi (k,n)}{n} \cdot \frac{\varphi(k,m)}{m}, \]
		and
		\[ \int_0^{\infty} \left( \sum_{n=1}^N \gamma_{n,k} (x) \right)^2 \, \mathrm{d}x = \sum_{n,m=1}^N \int_0^{\infty} \gamma_{n,k} (x) \gamma_{m,k}(x) \, \mathrm{d}x . \]
		The monotonicity assumptions made on $f_n$ ensure that the function $y \mapsto \int_0^{\infty} f_n (nx) f_m (mx -y) \, \mathrm{d}x$ is nondecreasing on $(-\infty, 0]$ and nonincreasing on $[0,\infty)$. Following the steps in the proof of \cite[Lemma 4]{SCH} verbatim thus yields
		\[ \int_0^{\infty} \gamma_{n,k} (x) \gamma_{m,k}(x) \, \mathrm{d}x \le I_n I_m + A(k,n,m) \int_0^{\infty} f_n (nx) f_m (mx) \, \mathrm{d}x, \]
		where $A(k,n,m)=|\{ (i,j) \in [0,m) \times [0,n) \, : \, ni-mj=0, \,\, \mathrm{gcd}(i,m) \le k, \,\, \mathrm{gcd}(j,n) \le k \}|$. Schmidt \cite[Lemma 5]{SCH} proved that $\sum_{n=1}^m A(k,n,m) \le m d_k (m)$, but the same argument shows that more generally, $\sum_{n=1}^{\ell} A(k,n,m) \le \ell d_k (m)$ for all positive integers $\ell$. By the previous three formulas,
		\[ \begin{split} \int_0^{\infty} \left( \sum_{n=1}^N \gamma_{n,k} (x) \right)^2 \, \mathrm{d}x - \left( \int_0^{\infty} \sum_{n=1}^N \gamma_{n,k}(x) \, \mathrm{d}x \right)^2 &\le \sum_{n,m=1}^N A (k,n,m) \int_0^{\infty} f_n (nx) f_m (mx) \, \mathrm{d}x \\ &\phantom{\le{}}+\sum_{n,m=1}^N I_n I_m \left( 1 - \frac{\varphi(k,n)}{n} \cdot \frac{\varphi (k,m)}{m} \right) . \end{split} \]
		The proof of the first claim of the lemma shows that here
		\[ \sum_{n,m=1}^N I_n I_m \left( 1 - \frac{\varphi(k,n)}{n} \cdot \frac{\varphi (k,m)}{m} \right) \le \sum_{n,m=1}^N I_n I_m \left( 1 - \frac{\varphi (k,n)}{n} + 1-\frac{\varphi (k,m)}{m} \right) \le 2C \frac{F(N)^2}{k} , \]
		as claimed.
	\end{proof}
	
	\begin{lem}\label{awaylemma} We have
		\[ \sum_{\substack{n=1 \\ \langle n \alpha \rangle \ge 1/2n}}^N \frac{1}{n \langle n \alpha \rangle} = \frac{1}{2} (\log N)^2 + \xi_{1,N}(\alpha) \quad \textrm{and} \quad \sum_{\substack{n=1 \\ \langle n \alpha \rangle \le -1/2n}}^N \frac{1}{n |\langle n \alpha \rangle|} = \frac{1}{2} (\log N)^2 + \xi_{2,N}(\alpha) \]
		with some error terms $\xi_{j,N}(\alpha)$, $j=1,2$ that satisfy
		\[ \mu \left( \frac{|\xi_{j,N}(\alpha)|}{\log N} \ge \frac{(\log \log N)^{1/3}}{(\log N)^{1/3}} \right) \ll \frac{(\log \log N)^{1/3}}{(\log N)^{1/3}} . \]
	\end{lem}
	
	\begin{proof} Assume first that $\alpha \sim \lambda$. Consider the sequences of functions
		\[ f_n (x) = \left\{ \begin{array}{ll} 2 & \text{if } 0 \le x < \frac{1}{2n}, \\ \frac{1}{nx} & \text{if } \frac{1}{2n} \le x \le \frac{1}{2}, \\ 0 & \text{else}, \end{array} \right. \qquad \textrm{and} \qquad \tilde{f}_n (x) = \left\{ \begin{array}{ll} 2 & \text{if } 0 \le x < \frac{1}{2n}, \\ 0 & \text{else}. \end{array} \right. \]
		Let $k \in [1,N]$ be an integer, to be chosen. Let $I_n$, $\gamma_n (x)$ and $\gamma_{n,k}(x)$ be as in Lemma \ref{schmidtlemma}, and let $\tilde{I}_n$, $\tilde{\gamma}_n (x)$ and $\tilde{\gamma}_{n,k}(x)$ be the same quantities in terms of $\tilde{f}_n$. Let
		\[ S_N(\alpha) = \sum_{n=1}^N \left( \gamma_n (\alpha) - \tilde{\gamma}_n(\alpha) \right) \quad \textrm{and} \quad S_{N,k} (\alpha) = \sum_{n=1}^N \left( \gamma_{n,k} (\alpha) - \tilde{\gamma}_{n,k} (\alpha) \right) . \]
		Observe that $f_n(x) - \tilde{f}_n(x) = \frac{1}{nx} \mathds{1}_{\{ 1/(2n) \le x \le 1/2 \}}$, thus
		\[ \sum_{\substack{n=1 \\ \langle n \alpha \rangle \ge 1/2n}}^N \frac{1}{n \langle n \alpha \rangle} = \sum_{n=1}^N \sum_{j=0}^{n-1} \frac{1}{n (n\alpha -j)} \mathds{1}_{\{ 1/(2n) \le n \alpha -j \le 1/2 \}} = S_N(\alpha) . \]
		
		The sequences $I_n = (1+\log n)/n$ and $\tilde{I}_n = 1/n$ are decreasing in $n$, hence condition \eqref{Ind} is satisfied with $C=1$. Lemma \ref{schmidtlemma} thus yields
		\[ \mathbb{E}_{\lambda} |S_N(\alpha) - S_{N,k} (\alpha)| \le \mathbb{E}_{\lambda} \left( \sum_{n=1}^N \gamma_n (\alpha) - \sum_{n=1}^N \gamma_{n,k}(\alpha) \right) + \mathbb{E}_{\lambda} \left( \sum_{n=1}^N \tilde{\gamma}_n (\alpha) - \sum_{n=1}^N \tilde{\gamma}_{n,k}(\alpha) \right) \ll \frac{(\log N)^2}{k}, \]
		and
		\[ \begin{split} \mathrm{Var}_{\lambda} S_{N,k}(\alpha) &\le \mathrm{Var}_{\lambda} \left( \sum_{n=1}^N \gamma_{n,k}(\alpha) \right) + \mathrm{Var}_{\lambda} \left( \sum_{n=1}^N \tilde{\gamma}_{n,k}(\alpha) \right) \\ &\ll \sum_{n,m=1}^N A(k,n,m) \left(  \int_0^{\infty} f_n (nx) f_m (mx) \, \mathrm{d}x +  \int_0^{\infty} \tilde{f}_n (nx) \tilde{f}_m (mx) \, \mathrm{d}x \right) + \frac{(\log N)^4}{k} . \end{split} \]
		For any $1 \le n \le m$, we have
		\[ \begin{split} \int_0^{\infty} f_n (nx) f_m (mx) \, \mathrm{d}x &= \left\{ \begin{array}{ll} \frac{4+4\log (m/n) - 2m/n^2}{m^2} & \textrm{if } m \le n^2, \\ \frac{2+2\log m}{m^2} & \textrm{if } m>n^2 \end{array} \right. \\ &\le \frac{4+4\log (m/n)}{m^2} \end{split} \]
		and $\int_0^{\infty} \tilde{f}_n (nx) \tilde{f}_m (mx) \, \mathrm{d}x = 2/m^2$. Summation by parts and the estimate $\sum_{n=1}^{\ell} A(k,n,m) \le \ell d_k (m)$ lead to
		\[ \begin{split} \sum_{n,m=1}^N A(k,n,m) &\left(  \int_0^{\infty} f_n (nx) f_m (mx) \, \mathrm{d}x +  \int_0^{\infty} \tilde{f}_n (nx) \tilde{f}_m (mx) \, \mathrm{d}x \right)  \\ &\ll \sum_{m=1}^N \sum_{n=1}^m A(k,n,m) \frac{1+\log (m/n)}{m^2} \\ &= \sum_{m=1}^N \frac{1}{m^2} \left( \sum_{\ell=1}^{m-1} \log \frac{\ell+1}{\ell} \sum_{n=1}^{\ell} A(k,n,m) + \sum_{n=1}^m A(k,n,m) \right) \\ &\ll \sum_{m=1}^N \frac{d_k (m)}{m} = \sum_{d=1}^k \sum_{\substack{m=1 \\ d \mid m}}^N \frac{1}{m} \ll \sum_{d=1}^k \frac{\log N}{d} \ll \log N \log k. \end{split} \]
		Choosing $k=\lfloor (\log N)^3 \rfloor$, the previous estimates show that $\mathbb{E}_{\lambda} |S_N(\alpha) - S_{N,k}(\alpha)| \ll 1$, and $\mathrm{Var}_{\lambda} S_{N,k}(\alpha) \ll \log N \log \log N$. In particular,
		\[ \mathbb{E}_{\lambda} S_{N,k}(\alpha) = \mathbb{E}_{\lambda} S_N (\alpha) + O(1) = \sum_{n=1}^N \frac{\log n}{n} + O(1) = \frac{1}{2} (\log N)^2 +O(1) . \]
		An application of the Markov and the Chebyshev inequalities yield
		\[ \begin{split} \lambda &\left( \frac{1}{\log N} \left| S_N(\alpha) - \frac{1}{2} (\log N)^2 \right| \ge \frac{(\log \log N)^{1/3}}{(\log N)^{1/3}} \right) \\ & \le \lambda \left( \frac{|S_N (\alpha) - S_{N,k}(\alpha)|}{\log N} \gg \frac{(\log \log N)^{1/3}}{(\log N)^{1/3}} \right) + \lambda \left( \frac{|S_{N,k} (\alpha) - \mathbb{E}_{\lambda} S_{N,k}(\alpha)|}{\log N} \gg \frac{(\log \log N)^{1/3}}{(\log N)^{1/3}} \right) \\ &\ll \frac{(\log \log N)^{1/3}}{(\log N)^{1/3}} . \end{split} \]
		Since $\alpha \mapsto 1-\alpha$ is a measure preserving map, we also have
		\[ \lambda \Bigg( \frac{1}{\log N} \bigg| \sum_{\substack{n=1 \\ \langle n \alpha \rangle \le -1/2n}}^N \frac{1}{n |\langle n \alpha \rangle|} - \frac{1}{2} (\log N)^2  \bigg| \ge \frac{(\log \log N)^{1/3}}{(\log N)^{1/3}} \Bigg) \ll \frac{(\log \log N)^{1/3}}{(\log N)^{1/3}} . \]
		This finishes the proof for $\alpha \sim \lambda$.
		
		The density $\mathrm{d}\mu / \mathrm{d}\lambda$ is bounded above by a constant depending only on $L$, hence the previous two tail estimates remain true for any $\mu$ that satisfies \eqref{lipschitz}.
	\end{proof}
	
	\subsection{Close to the singularity}\label{closesection}
	
	We start with a tail estimate for sums of partial quotients.
	\begin{lem}\label{sumaktailslemma} For any integers $M \ge 0$ and $K \ge 1$, and any real $t>2 \log K$,
		\[ \mu \left( \sum_{k=M+1}^{M+K} a_k \ge t K \right) \ll \frac{1}{t} . \]
	\end{lem}
	
	\begin{proof} Diamond and Vaaler \cite[Eq.\ (6)]{DV} showed that for any $t>0$,
		\[ \mu_{\mathrm{Gauss}} \left( \left| \sum_{k=1}^K a_k - \frac{K \log K}{\log 2} \right| \ge t K \right) \ll \frac{1}{t} . \]
		Since the partial quotients are strictly stationary, and the assumption $t>2 \log K$ implies that $(K \log K) / \log 2 \le (3/4) tK$,
		\[ \mu_{\mathrm{Gauss}} \left( \sum_{k=M+1}^{M+K} a_k \ge t K \right) \le \mu_{\mathrm{Gauss}} \left( \left| \sum_{k=M+1}^{M+K} a_k - \frac{K \log K}{\log 2} \right| \ge \frac{t K}{4} \right) \ll \frac{1}{t}. \]
		The density $\mathrm{d}\mu / \mathrm{d}\mu_{\mathrm{Gauss}}$ is bounded above by a constant depending only on $L$, hence the previous tail estimate remains true for any $\mu$ that satisfies \eqref{lipschitz}.
	\end{proof}
	
	For any $N \ge 1$, let $K_N^*=K_N^*(\alpha)$ denote the random index for which $q_{K_N^*} \le N < q_{K_N^*+1}$, and let $K_N$ be a deterministic even integer such that $|K_N - \frac{12 \log 2}{\pi^2} \log N| \le 1$.
	\begin{lem}\label{KNlemma} There exists a universal constant $\tau>0$ such that
		\[ \mu \left( \left| K_N^* - K_N \right| \ge \tau (\log N \log \log N)^{1/2} \right) \ll \frac{1}{(\log N)^{1/2}} . \]
	\end{lem}
	
	\begin{proof} In the special case $\alpha \sim \lambda$, the convergent denominators satisfy the central limit theorem with rate
		\[ \lambda \left( \frac{\log q_n - \frac{\pi^2}{12 \log 2} n}{\tau_1 n^{1/2}} \le x \right) = \int_{-\infty}^x \frac{e^{-y^2/2}}{(2 \pi)^{1/2}} \, \mathrm{d}y + O \left( \frac{1}{n^{1/2}} \right) \]
		with a universal constant $\tau_1>0$, see Morita \cite{MO}, Vall\'ee \cite{VA} and \cite[p.\ 194--195]{IK}. In particular,
		\[ \lambda \left( \left| \log q_n - \frac{\pi^2}{12 \log 2} n \right| \ge 10 \tau_1 (n \log n)^{1/2} \right) \ll \frac{1}{n^{1/2}} , \]
		and hence
		\[ \lambda \left( \left| \log q_{K_N} - \log N \right| \ge \tau_2 (\log N \log \log N)^{1/2} \right) \ll \frac{1}{(\log N)^{1/2}} \]
		with a suitably large universal constant $\tau_2>0$. The general fact $q_{k+2}/q_k \ge 2$ for all $k \ge 0$ implies that $|\log q_{K_N^*} - \log q_{K_N}| \gg |K_N^* -K_N|$, thus
		\[ \lambda \left( \left| K_N^* - K_N \right| \ge \tau_3 (\log N \log \log N)^{1/2} \right) \le \lambda \left( \left| \log q_{K_N} - \log N \right| \ge \tau_2 (\log N \log \log N)^{1/2} \right) \]
		with a large enough universal constant $\tau_3 >0$, and the claim follows for $\alpha \sim \lambda$.
		
		The density $\mathrm{d}\mu / \mathrm{d}\lambda$ is bounded above by a constant depending only on $L$, hence the tail estimate in the claim remains true for any $\mu$ that satisfies \eqref{lipschitz}.
	\end{proof}
	
	Let
	\[ u_k = \frac{1}{q_{k-1} \| q_{k-1} \alpha \|}, \,\,\, k \ge 1 \qquad \textrm{and} \qquad R(x) = \sum_{1 \le j < (x/2)^{1/2}} \frac{x}{j^2}, \,\,\, x>0. \]
	By a classical identity of continued fractions,
	\begin{equation}\label{uk}
		u_k = [a_k ; a_{k+1}, a_{k+2}, \ldots] + [0;a_{k-1}, a_{k-2}, \ldots, a_1] .
	\end{equation}
	In particular, $a_k < u_k < a_k+2$ for all $k \ge 1$. The graph of $R(x)$ consists of straight line segments with a jump of size $2$ at the points $2m^2$, $m \in \mathbb{N}$. More precisely, let $r_m=\sum_{j=1}^m 1/j^2$ with the convention $r_0=0$. For any integer $m \ge 0$, we have $R(x)=r_m x$ on the interval $x \in (2m^2, 2(m+1)^2]$. In particular, $R(x)=(\pi^2 /6) x + O(x^{1/2})$.
	\begin{lem}\label{Ruklemma} We have
		\[ \sum_{\substack{n=1 \\ 0< \langle n \alpha \rangle < 1/(2n)}}^N \frac{1}{n \langle n \alpha \rangle} = \sum_{\substack{k=1 \\ k \textrm{ odd}}}^{K_N} R(u_k) + \xi_{1,N}(\alpha) \quad \textrm{and} \sum_{\substack{n=1 \\ -1/(2n)< \langle n \alpha \rangle < 0}}^N \frac{1}{n |\langle n \alpha \rangle|} = \sum_{\substack{k=1 \\ k \textrm{ even}}}^{K_N} R(u_k) + \xi_{2,N}(\alpha) \]
		with some error terms $\xi_{j,N}(\alpha)$, $j=1,2$ that satisfy
		\[ \mu \left( \frac{|\xi_{j,N}(\alpha)|}{\log N} \ge \frac{(\log \log N)^{1/4}}{(\log N)^{1/4}} \right) \ll \frac{(\log \log N)^{1/4}}{(\log N)^{1/4}} . \]
	\end{lem}
	
	\begin{proof} Let $k \ge 0$ and $q_k \le n < q_{k+1}$ such that $\| n \alpha \| < 1/(2n)$. Legendre's theorem on Diophantine approximation \cite[p.\ 30]{KH} implies that $q_k \mid n$. A multiple $n=j q_k$ satisfies $\| j q_k \alpha \| < 1/(2jq_k)$ if and only if $1 \le j < (2 q_k \| q_k \alpha \|)^{-1/2}$. Since $\mathrm{sgn}(q_k \alpha - p_k) = (-1)^k$, we have $\langle j q_k \alpha \rangle >0$ if $k$ is even, and $\langle j q_k \alpha \rangle <0$ if $k$ is odd. Therefore
		\[ \begin{split} \sum_{\substack{q_k \le n < q_{k+1} \\ 0< \langle n \alpha \rangle < 1/(2n)}} \frac{1}{n \langle n \alpha \rangle} &= \mathds{1}_{\{ k \textrm{ even} \}} \sum_{1 \le j < (2q_k \| q_k \alpha \|)^{-1/2}} \frac{1}{jq_k \| j q_k \alpha \|} = \mathds{1}_{\{ k \textrm{ even} \}} R(u_{k+1}), \\ \sum_{\substack{q_k \le n < q_{k+1} \\ -1/(2n)< \langle n \alpha \rangle < 0}} \frac{1}{n |\langle n \alpha \rangle|} &= \mathds{1}_{\{ k \textrm{ odd} \}} \sum_{1 \le j < (2q_k \| q_k \alpha \|)^{-1/2}} \frac{1}{jq_k \| j q_k \alpha \|} = \mathds{1}_{\{ k \textrm{ odd} \}} R(u_{k+1}) . \end{split} \]
		For each $k \ge 0$ we have $R(u_{k+1}) \ll a_{k+1}$, hence summing over $k$ leads to the desired formulas with error terms
		\[ |\xi_{j,N} (\alpha)| \ll \sum_{|k-K_N| \le |K_N^*-K_N|+2} a_k, \qquad j=1,2. \]
		Lemma \ref{KNlemma} shows that outside a set of $\mu$-measure $\ll (\log N)^{-1/2}$, the previous sum has $\ll (\log N \log \log N)^{1/2}$ terms. We conclude by an application of Lemma \ref{sumaktailslemma} with $K \approx (\log N \log \log N)^{1/2}$ and $t=(\log N)^{1/4} (\log \log N)^{-1/4}$.
	\end{proof}
	
	Our next goal is to replace $u_k$ by $a_k$ in the sums $\sum_k R(u_k)$. We follow the approach of Samur \cite{SA1}, who considered the same problem under a suitable assumption on the modulus of continuity of the function. Certain modifications are thus needed to handle the discontinuous function $R$. Recall that $\mathcal{A}_{\ell}^k$ denotes the $\sigma$-algebra generated by the partial quotients $a_j$, $\ell \le j \le k$. We use the convention $\mathcal{A}_{\ell}^k = \mathcal{A}_1^k$ if $\ell \le 0$.
	
	\begin{lem}\label{conditionallemma} Let $\mu' \ll \lambda$ be another Borel probability measure on $[0,1]$ that satisfies \eqref{lipschitz}. For any $k \ge 1$ and $\ell \ge 0$,
		\[ \mathbb{E}_{\mu} \left( R(u_k) - R(a_k) - \mathbb{E}_{\mu'} \left( R(u_k) - R(a_k) \mid \mathcal{A}_{k-\ell}^{k+\ell} \right) \right)^2 \ll 2^{-\ell} . \]
	\end{lem}
	
	\begin{proof} First, assume that $k-\ell \ge 1$. Let $b_{k-\ell}, b_{k-\ell+1}, \ldots, b_{k+\ell} \in \mathbb{N}$ be arbitrary, and consider the set $S = \{ \alpha \in [0,1] \, : \, a_{k-\ell}= b_{k-\ell}, \ldots, a_{k+\ell} = b_{k+\ell}  \}$. Further, let $I$ be the set of all continued fractions of the form $[0;b_{k+1}, \ldots, b_{k+\ell}, \ldots]$, and let $J$ be the set of all continued fractions of the form $[0; b_{k-1}, \ldots, b_{k-\ell}, \ldots]$. Then $I,J \subseteq [0,1]$ are intervals with rational endpoints of length $\le 2^{-\ell}$ each.
		
		Now let $\alpha \in S$. Formula \eqref{uk} shows that $u_k \in b_k+I+J = \{ b_k +x+y \, : \, x \in I, \, y \in J \}$. If $1 \not\in I+J$, then the interior of $b_k+I+J$ does not contain any integer. In particular, $R$ does not have a jump in the interior of $b_k+I+J$, and hence $\sup_{\alpha \in S} R(u_k) - \inf_{\alpha \in S} R(u_k) \le (\pi^2/ 3) 2^{-\ell}$. If $1 \in I +J$, then $R$ has at most one jump (of size $2$) in the interior of $b_k+I+J$, and hence $\sup_{\alpha \in S} R(u_k) - \inf_{\alpha \in S} R(u_k) \le (\pi^2/ 3) 2^{-\ell}+2$. Since $R(a_k)=R(b_k)$ is constant on $S$, we obtain
		\[ \sup_{\alpha \in S} (R(u_k) - R(a_k)) - \inf_{\alpha \in S} (R(u_k) - R(a_k)) \le \frac{\pi^2}{3} 2^{-\ell} + \mathds{1}_{\{ 1 \in I+J \}} 2 . \]
		Consequently,
		\[ \left| R(u_k) - R(a_k) - \mathbb{E}_{\mu'} \left( R(u_k) - R(a_k) \mid \mathcal{A}_{k-\ell}^{k+\ell} \right) \right| \le \frac{\pi^2}{3} 2^{-\ell} + \mathds{1}_{\{ 1 \in I+J \}} 2, \quad \alpha \in S, \]
		and so
		\[  \mathbb{E}_{\mu} \left( R(u_k) - R(a_k) - \mathbb{E}_{\mu'} \left( R(u_k) - R(a_k) \mid \mathcal{A}_{k-\ell}^{k+\ell} \right) \right)^2 \ll 2^{-\ell} + \sum_{b_{k-\ell}, \ldots, b_{k+\ell} \in \mathbb{N}} \mathds{1}_{\{ 1 \in I+J \}} \mu (S). \]
		A simple consequence of the mixing properties of the partial quotients is that here $\mu (S) \ll \mu (H) \mu (I)$ with $H= \{ \alpha \in [0,1] \, : \, a_1 = b_{k-\ell}, \, a_2 = b_{k-\ell+1}, \ldots, a_{\ell +1} = b_k \}$. For fixed $b_{k-\ell}, \ldots, b_k$, the total $\mu$-measure of all intervals $I$ that intersect the given interval $1-J$ of length $\le 2^{-\ell}$ is $\ll 2^{-\ell}$. Thus
		\[ \sum_{b_{k-\ell}, \ldots, b_{k+\ell} \in \mathbb{N}} \mathds{1}_{\{ 1 \in I+J \}} \mu (S) \ll \sum_{b_{k-\ell}, \ldots, b_k \in \mathbb{N}} \mu (H) 2^{-\ell} = 2^{-\ell}, \]
		and the claim follows.
		
		The proof in the case $k-\ell <1$ is similar, the only difference is that we use $b_1, \ldots, b_{k+\ell} \in \mathbb{N}$ instead of $b_{k-\ell}, \ldots, b_{k+\ell} \in \mathbb{N}$.
	\end{proof}
	
	\begin{lem}\label{expectedvaluelemma} For any $k \ge 1$, we have $\mathbb{E}_{\mu} (R(u_k) - R(a_k)) = \kappa + O (e^{-ak})$ with the constant
		\begin{equation}\label{kappa}
			\kappa = \frac{\pi^2}{6 \log 2} + \frac{1}{\log 2} \sum_{m=1}^{\infty} \left( \frac{1}{m^2} \log \frac{2m^2 +1}{2m^2} - 2 \log \frac{2m^2 +2}{2m^2 +1} \right) .
		\end{equation}
	\end{lem}
	
	\begin{proof} Let $\ell = \lfloor k/2 \rfloor$, and consider $X_k = R(u_k) - R(a_k)$ and $Y_k =\mathbb{E}_{\mu_{\mathrm{Gauss}}} (R(u_k) - R(a_k) \mid \mathcal{A}_{k-\ell}^{k+\ell})$. Lemma \ref{conditionallemma} shows that
		\[ | \mathbb{E}_{\mu} X_k - \mathbb{E}_{\mu} Y_k | \le \left( \mathbb{E}_{\mu} (X_k-Y_k)^2 \right)^{1/2} \ll 2^{-k/4} . \]
		Since $Y_k$ is bounded and $\mathcal{A}_{k-\ell}^{\infty}$-measurable, an application of the Gauss--Kuzmin theorem \eqref{gausskuzmin} gives $|\mathbb{E}_{\mu} Y_k - \mathbb{E}_{\mu_{\mathrm{Gauss}}} Y_k| \ll e^{-ak}$. Here $\mathbb{E}_{\mu_{\mathrm{Gauss}}} Y_k = \mathbb{E}_{\mu_{\mathrm{Gauss}}} X_k$, hence
		\[ \mathbb{E}_{\mu} (R(u_k) - R(a_k)) = \mathbb{E}_{\mu_{\mathrm{Gauss}}} (R(u_k)-R(a_k)) + O(e^{-ak}) . \]
		In particular, $\mathbb{E}_{\mu} (R(u_k) - R(a_k))$ has the same value for any $\mu \ll \lambda$ that satisfies \eqref{lipschitz} up to an error $O(e^{-ak})$. It is thus enough to compute the expected value with respect to $\lambda$.
		
		Fix $b_1, b_2, \ldots, b_{k-1} \in \mathbb{N}$, and consider $I=\{ \alpha \in [0,1] \, : \, a_1 = b_1, \ldots, a_{k-1} = b_{k-1} \}$. The endpoints of the interval $I$ are $(p_{k-2}+p_{k-1}) / (q_{k-2} + q_{k-1})$ and $p_{k-1}/q_{k-1}$, and it has length $\lambda (I) = 1/(q_{k-1} (q_{k-1} + q_{k-2}))$. Let us write $\alpha \in I$ in the form $\alpha = [0;b_1, \ldots, b_{k-1}, x] = (p_{k-1} x + p_{k-2}) / (q_{k-1} x + q_{k-2})$ with $x=[a_k;a_{k+1}, \ldots]$, and let $y_k = q_{k-2}/q_{k-1} = [0;b_{k-1}, b_{k-2}, \ldots, b_1]$. Then $u_k = x +y_k$ by \eqref{uk} and $a_k=\lfloor x \rfloor$, hence
		\[ \int_I (R(u_k) - R(a_k)) \, \mathrm{d} \lambda (\alpha) = \int_1^{\infty} \frac{R(x+y_k) - R(\lfloor x \rfloor)}{(q_{k-1} x + q_{k-2})^2} \, \mathrm{d}x =\lambda (I) (1+y_k) \int_1^{\infty} \frac{R(x+y_k)-R(\lfloor x \rfloor)}{(x+y_k)^2} \, \mathrm{d}x . \]
		Therefore
		\begin{equation}\label{conditionalH}
			\mathbb{E}_{\lambda} \left( R(u_k) - R(a_k) \mid [0; a_{k-1}, a_{k-2}, \ldots, a_1] =y \right) = W(y)
		\end{equation}
		with the function
		\[ W(y) = (1+y) \int_1^{\infty} \frac{R(x+y)-R(\lfloor x \rfloor)}{(x+y)^2} \, \mathrm{d}x, \qquad y \in [0,1] . \]
		
		We proceed by computing the function $W$. Recall that $R(x)=r_m x$ on the interval $x \in (2m^2, 2(m+1)^2]$, where $r_m = \sum_{j=1}^m 1/j^2$. We first integrate on the interval $x \in (1,2)$, where $R(\lfloor x \rfloor) =R(1)=0$ and $R(x+y) = \mathds{1}_{(2-y,2)} (x)(x+y)$:
		\begin{equation}\label{integral12}
			\int_1^2 \frac{R(x+y)-R(\lfloor x \rfloor)}{(x+y)^2} \, \mathrm{d}x = \int_{2-y}^2 \frac{1}{x+y} \, \mathrm{d}x = \log \frac{2+y}{2} .
		\end{equation}
		Now let $m \ge 1$ and $2m^2 \le n < 2(m+1)^2$, and let us integrate on the interval $x \in (n,n+1)$. If $2m^2 +1 \le n \le 2(m+1)^2-2$, then $R(\lfloor x \rfloor) = R(n) = r_m n$ and $R(x+y) = r_m (x+y)$, thus
		\[ \int_n^{n+1} \frac{R(x+y)-R(\lfloor x \rfloor)}{(x+y)^2} \, \mathrm{d}x = r_m \left( \log \frac{n+1+y}{n+y} - \frac{n}{(n+y)(n+1+y)} \right) . \]
		If $n=2m^2$, then $R(\lfloor x \rfloor) = R(n) = r_{m-1} n = r_m n - 2$ and $R(x+y) = r_m (x+y)$, thus
		\[ \int_n^{n+1} \frac{R(x+y)-R(\lfloor x \rfloor)}{(x+y)^2} \, \mathrm{d}x = r_m \left( \log \frac{n+1+y}{n+y} - \frac{n}{(n+y)(n+1+y)} \right) + \frac{2}{(2m^2 +y)(2m^2 +1+y)} . \]
		If $n=2(m+1)^2-1$, then $R(\lfloor x \rfloor) = R(n) = r_m n$ and $R(x+y)=r_m (x+y)$ on $x \in (n,n+1-y)$ and $R(x+y) = r_{m+1}(x+y) = r_m (x+y) + (x+y)/(m+1)^2$ on $x \in (n+1-y,n+1)$, thus
		\[ \begin{split} \int_n^{n+1} \frac{R(x+y)-R(\lfloor x \rfloor)}{(x+y)^2} \, \mathrm{d}x &= r_m \left( \log \frac{n+1+y}{n+y} - \frac{n}{(n+y)(n+1+y)} \right) \\ &\phantom{={}} + \frac{1}{(m+1)^2} \log \frac{2(m+1)^2 +y}{2(m+1)^2} . \end{split} \]
		Summing the previous three formulas over $2m^2 \le n < 2(m+1)^2$ yields
		\[ \begin{split} \int_{2m^2}^{2(m+1)^2} \frac{R(x+y)-R(\lfloor x \rfloor)}{(x+y)^2} \, \mathrm{d}x &= r_m \sum_{n=2m^2}^{2(m+1)^2-1} \left( \log \frac{n+1+y}{n+y} - \frac{n}{(n+y)(n+1+y)} \right) \\ &\phantom{={}} +\frac{2}{(2m^2 +y)(2m^2 +1+y)} + \frac{1}{(m+1)^2} \log \frac{2(m+1)^2 +y}{2(m+1)^2} . \end{split} \]
		Summation by parts then leads to
		\[ \begin{split} \int_{2}^{\infty} \frac{R(x+y)-R(\lfloor x \rfloor)}{(x+y)^2} \, \mathrm{d}x &= \sum_{m=1}^{\infty} \frac{-1}{(m+1)^2} \sum_{n=2}^{2(m+1)^2 -1} \left( \log \frac{n+1+y}{n+y} - \frac{n}{(n+y)(n+1+y)} \right) \\ &\phantom{={}}+\frac{\pi^2}{6} \sum_{n=2}^{\infty} \left( \log \frac{n+1+y}{n+y} - \frac{n}{(n+y)(n+1+y)} \right) \\ &\phantom{={}}+\sum_{m=1}^{\infty} \frac{2}{(2m^2 +y)(2m^2 +1+y)} + \sum_{m=2}^{\infty} \frac{1}{m^2} \log \frac{2m^2 +y}{2m^2} . \end{split} \]
		The sums over $n$ are telescoping:
		\[ \sum_{n=2}^{N-1} \left( \log \frac{n+1+y}{n+y} - \frac{n}{(n+y)(n+1+y)} \right) = \log \frac{N+y}{2+y} + \frac{N}{N+y} - \frac{1}{2+y} - \sum_{n=2}^N \frac{1}{n+y}. \]
		Using the relation $\lim_{N\to \infty} (\log (N+y) - \sum_{n=2}^N 1/(n+y)) =\Gamma'(1+y)/\Gamma (1+y) + 1/(1+y)$, after some simplification we arrive at
		\[ \begin{split} &\sum_{m=1}^{\infty} \frac{-1}{(m+1)^2} \sum_{n=2}^{2(m+1)^2 -1} \left( \log \frac{n+1+y}{n+y} - \frac{n}{(n+y)(n+1+y)} \right) \\ &+\frac{\pi^2}{6} \sum_{n=2}^{\infty} \left( \log \frac{n+1+y}{n+y} - \frac{n}{(n+y)(n+1+y)} \right) \\ &= \sum_{m=1}^{\infty} \left( -\frac{\log (2m^2+y)}{m^2} - \frac{2}{2m^2+y} + \frac{1}{m^2} \sum_{n=2}^{2m^2} \frac{1}{n+y} \right) + \frac{\pi^2}{6} \left( 1+\frac{\Gamma'(1+y)}{\Gamma (1+y)}+\frac{1}{1+y} \right) . \end{split} \]
		Together with \eqref{integral12}, we thus obtain
		\[ \begin{split} \int_{1}^{\infty} \frac{R(x+y)-R(\lfloor x \rfloor)}{(x+y)^2} \, \mathrm{d}x &= \sum_{m=1}^{\infty} \left( - \frac{\log (2m^2)}{m^2} - \frac{2}{2m^2+1+y} + \frac{1}{m^2} \sum_{n=2}^{2m^2} \frac{1}{n+y} \right) \\ &\phantom{={}}+\frac{\pi^2}{6} \left( 1+\frac{\Gamma'(1+y)}{\Gamma (1+y)}+\frac{1}{1+y} \right) , \end{split} \]
		which yields an explicit formula for the function $W(y)$. One readily checks that $W$ is $C^1$, and hence is in particular of bounded variation on $[0,1]$.
		
		Formula \eqref{conditionalH} together with the byproduct of L\'evy's solution to the Gauss--Kuzmin problem \cite[p.\ 41]{IK}
		\[ \sup_{y \in [0,1]} \left| \lambda ([0;a_{k-1}, \ldots, a_1] \le y) - \mu_{\mathrm{Gauss}}([0,y]) \right| \ll e^{-ak} \]
		thus yield the desired formula
		\[ \begin{split} \mathbb{E}_{\lambda} (R(u_k) - R(a_k)) &= \mathbb{E}_{\lambda} W([0;a_{k-1}, \ldots, a_1]) = \frac{1}{\log 2} \int_0^1 \frac{W(y)}{1+y} \, \mathrm{d} y + O(e^{-ak}) \\ &= \frac{1}{\log 2} \sum_{m=1}^{\infty} \left( \frac{1}{m^2} \log \frac{2m^2 +1}{2m^2} - 2 \log \frac{2m^2 +2}{2m^2 +1} \right) + \frac{\pi^2}{6\log 2} + O (e^{-ak}). \end{split} \]
	\end{proof}
	
	\begin{lem}\label{uktoaklemma} We have
		\[ \sum_{\substack{k=1 \\ k \textrm{ odd}}}^{K_N} R(u_k) = \sum_{\substack{k=1 \\ k \textrm{ odd}}}^{K_N} R(a_k) + \frac{\kappa}{2} K_N + \xi_{1,N}(\alpha) \quad \textrm{and} \quad \sum_{\substack{k=1 \\ k \textrm{ even}}}^{K_N} R(u_k) = \sum_{\substack{k=1 \\ k \textrm{ even}}}^{K_N} R(a_k) + \frac{\kappa}{2} K_N + \xi_{2,N}(\alpha) \]
		with the constant $\kappa$ defined in \eqref{kappa} and some error terms $\xi_{j,N}(\alpha)$, $j=1,2$ that satisfy
		\[ \mu \left( \frac{|\xi_{j,N}(\alpha)|}{\log N} \ge \frac{1}{(\log N)^{1/3}} \right) \ll \frac{1}{(\log N)^{1/3}} . \]
	\end{lem}
	
	\begin{proof} Let $\eta_k = R(u_k)-R(a_k) - \mathbb{E}_{\mu} (R(u_k)-R(a_k))$. Clearly, these are mean zero, and uniformly bounded. By Lemma \ref{expectedvaluelemma}, the desired formulas hold with error terms
		\[ \xi_{1,N}(\alpha) = \sum_{\substack{k=1 \\ k \textrm{ odd}}}^{K_N} \eta_k + O(1) \quad \textrm{and} \quad \xi_{2,N}(\alpha) = \sum_{\substack{k=1 \\ k \textrm{ even}}}^{K_N} \eta_k + O(1) . \]
		
		Fix $1 \le j \le k$, and let $\ell = \lfloor (k-j)/3 \rfloor$. Consider $\eta_j^*=\mathbb{E}_{\mu} (\eta_j \mid \mathcal{A}_{j-\ell}^{j+\ell})$ and $\eta_k^*=\mathbb{E}_{\mu} (\eta_k \mid \mathcal{A}_{k-\ell}^{k+\ell})$. An application of Lemma \ref{conditionallemma} yields $\mathbb{E}_{\mu} (\eta_j - \eta_j^*)^2 \ll 2^{-(k-j)/3}$ and $\mathbb{E}_{\mu} (\eta_k - \eta_k^*)^2 \ll 2^{-(k-j)/3}$. On the other hand, $\eta_j^*$ and $\eta_k^*$ are functions of partial quotients whose indices are separated by at least $(k-j)/3$, hence by exponentially fast $\psi$-mixing, $|\mathbb{E}_{\mu} \eta_j^* \eta_k^* | \ll e^{-a (k-j)}$. Therefore $|\mathbb{E}_{\mu} \eta_j \eta_k | \ll e^{-a(k-j)}$, and
		\[ \mathbb{E}_{\mu} \Bigg( \sum_{\substack{k=1 \\ k \textrm{ odd}}}^{K_N} \eta_k \Bigg)^2 \le 2 \sum_{1 \le j \le k \le K_N} |\mathbb{E}_{\mu} \eta_j \eta_k|  \ll K_N \ll \log N. \]
		The same holds with ``odd'' replaced by ``even''. We conclude by an application of the Chebyshev inequality.
	\end{proof}
	
	In the last two lemmas of this section, we find the limit distribution of $\sum_{k} R(a_k)$. We follow the approach of Heinrich \cite{HE}.
	\begin{lem}\label{gausstransformationlemma} For any $t_1, t_2 \in \mathbb{R}$,
		\[ \begin{split} \mathbb{E}_{\mu_{\mathrm{Gauss}}} \left( e^{i (t_1 R(a_1) + t_2 R(a_2))} -1 \right) &= - \frac{\pi^3}{12 \log 2} |t_1| - i \frac{\pi^2}{6 \log 2} t_1 \log |t_1| - i \kappa' t_1 \\ &\phantom{={}} - \frac{\pi^3}{12 \log 2} |t_2| - i \frac{\pi^2}{6 \log 2} t_2 \log |t_2| - i \kappa' t_2 +O \left( |t_1|^{3/2} + |t_2|^{3/2} \right) \end{split} \]
		with the constant
		\begin{equation}\label{kappa'}
			\kappa' = \frac{\pi^2}{6 \log 2} \left( \gamma + \log \frac{\pi^2}{6} \right) + \frac{1}{\log 2} \sum_{m=1}^{\infty} \left( \frac{\log (2m^2 +1)}{m^2} - 2 \log \frac{2m^2 +2}{2m^2 +1} \right) .
		\end{equation}
	\end{lem}
	
	\begin{proof} We may assume that $|t_1|, |t_2| \le 1/2$, otherwise the error term in the claim is larger than the main term. Recall that $R(x)=(\pi^2 /6) x + O(x^{1/2})$.
		
		Let us write
		\[ e^{i (t_1 R(a_1) + t_2 R(a_2))} -1 = \left( e^{i t_1 R(a_1)} -1 \right) \left( e^{i t_2 R(a_2)} -1 \right) + e^{it_1 R(a_1)} -1 + e^{it_2 R(a_2)} -1 . \]
		Since $\mu_{\mathrm{Gauss}} (a_1 = b_1, \, a_2 = b_2) \ll b_1^{-2} b_2^{-2}$, we have
		\[ \begin{split} \left| \mathbb{E}_{\mu_{\mathrm{Gauss}}} \left( e^{i t_1 R(a_1)} -1 \right) \left( e^{i t_2 R(a_2)} -1 \right) \right| &\ll \sum_{b_1, b_2 =1}^{\infty} \frac{|e^{it_1 R(b_1)} -1| \cdot |e^{it_2 R(b_2)}-1|}{b_1^2 b_2^2} \\ &\ll \sum_{b_1=1}^{\infty} \frac{\min \{ 1, |t_1| b_1 \}}{b_1^2} \sum_{b_2=1}^{\infty} \frac{\min \{ 1, |t_2| b_2 \}}{b_2^2} \\ &\ll |t_1| \log \frac{1}{|t_1|} \cdot |t_2| \log \frac{1}{|t_2|} \\ &\ll |t_1|^2 \log^2 \frac{1}{|t_1|} + |t_2|^2  \log^2 \frac{1}{|t_2|} . \end{split} \]
		As $a_1$ and $a_2$ are identically distributed, the previous two formulas yield
		\[ \mathbb{E}_{\mu_{\mathrm{Gauss}}} \left( e^{i (t_1 R(a_1) + t_2 R(a_2))} -1 \right) = w(t_1) + w(t_2) + O\left( |t_1|^2 \log^2 \frac{1}{|t_1|}+ |t_2|^2 \log^2 \frac{1}{|t_2|} \right) \]
		with the same function $w(t)=\mathbb{E}_{\mu_{\mathrm{Gauss}}} (e^{itR(a_1)} -1)$. It remains to estimate $w(t)$.
		
		By \cite[Corollary 3.3]{BD}, the same expectation with $R(a_1)$ replaced by $(\pi^2 /6) a_1$ is
		\begin{equation}\label{hwitha1}
			\mathbb{E}_{\mu_{\mathrm{Gauss}}} (e^{it (\pi^2 /6) a_1} -1) = - \frac{\pi^3}{12 \log 2} |t| - i \frac{\pi^2}{6 \log 2} t \log |t| - i \frac{\pi^2}{6 \log 2} \left( \gamma + \log \frac{\pi^2}{6} \right) t + O \left( t^2 \log \frac{1}{|t|} \right) .
		\end{equation}
		Consider now
		\[ \mathbb{E}_{\mu_{\mathrm{Gauss}}} \left( e^{it R(a_1)} - e^{it (\pi^2 /6) a_1} \right) = \frac{1}{\log 2} \sum_{n=1}^{\infty} e^{i t (\pi^2 /6) n} \left( e^{i t (R(n) - (\pi^2 /6)n )} -1 \right) \log \left( 1+\frac{1}{n (n+2)} \right) . \]
		The contribution of the terms $n > 1/t^2$ is $O(t^2)$. The Taylor expansion $e^{it (R(n)-(\pi^2 /6) n)}-1 = it (R(n) - (\pi^2 /6)n) + O(t^2 n)$ for the terms $n \le 1/t^2$ leads to
		\begin{multline*}
			\mathbb{E}_{\mu_{\mathrm{Gauss}}} \left( e^{it R(a_1)} - e^{it (\pi^2 /6) a_1} \right) \\ = \frac{it}{\log 2} \sum_{1 \le n \le 1/t^2} e^{i t (\pi^2 /6) n} (R(n) - (\pi^2 /6)n) \log \left( 1+\frac{1}{n (n+2)} \right) + O \left( t^2 \log \frac{1}{|t|} \right) .
		\end{multline*}
		Here the contribution of the terms $1/|t| < n \le 1/t^2$ is $O(|t|^{3/2})$. Using $e^{it (\pi^2 /6)n} = 1+O(|t| n)$ for the terms $n \le 1/|t|$, and finally extending the range of summation yields
		\[ \mathbb{E}_{\mu_{\mathrm{Gauss}}} \left( e^{it R(a_1)} - e^{it (\pi^2 /6) a_1} \right) = \frac{it}{\log 2} \sum_{n=1}^{\infty} (R(n) - (\pi^2 /6)n) \log \left( 1+\frac{1}{n (n+2)} \right) + O(|t|^{3/2}) . \]
		The previous formula together with \eqref{hwitha1} shows that
		\[ w(t) = - \frac{\pi^3}{12 \log 2} |t| - i \frac{\pi^2}{6 \log 2} t \log |t| - i \kappa' t + O(|t|^{3/2}) \]
		with the constant
		\[ \kappa'=\frac{\pi^2}{6 \log 2} \left( \gamma + \log \frac{\pi^2}{6} \right) - \frac{1}{\log 2} \sum_{n=1}^{\infty} (R(n) - (\pi^2 /6)n) \log \left( 1+\frac{1}{n (n+2)} \right) , \]
		which proves the claim of the lemma up to the value of the constant $\kappa'$.
		
		Finally, we show that $\kappa'$ is as in \eqref{kappa'}. For any integers $m \ge 0$ and $2m^2 < n \le 2(m+1)^2$, we have $R(n)=r_m n$. Summation by parts leads to
		\[ \begin{split} \sum_{n=1}^{\infty} (R(n) - (\pi^2 /6)n) \log \left( 1+\frac{1}{n (n+2)} \right) &= \sum_{m=0}^{\infty} \left( r_m - \frac{\pi^2}{6} \right) \sum_{n=2m^2 +1}^{2(m+1)^2} n \log \left( 1+\frac{1}{n(n+2)} \right) \\ &= \sum_{m=0}^{\infty} \frac{-1}{(m+1)^2} \sum_{n=1}^{2(m+1)^2} n \log \left( 1+\frac{1}{n(n+2)} \right) . \end{split} \]
		Writing
		\[ n \log \left( 1+\frac{1}{n(n+2)} \right) = (n+1) \log (n+1) - n \log n + (n-1) \log (n+1) - n \log (n+2) , \]
		we observe that the sum over $n$ is telescoping, hence
		\[ \sum_{n=1}^{\infty} (R(n) - (\pi^2 /6)n) \log \left( 1+\frac{1}{n (n+2)} \right) = \sum_{m=1}^{\infty} \frac{-1}{m^2} \left( \log (2m^2 +1) - 2m^2 \log \frac{2m^2 +2}{2m^2 +1} \right) . \]
		This proves the form of the constant $\kappa'$ stated in \eqref{kappa'}.
	\end{proof}
	
	We will need a suitable form of the two-dimensional Esseen inequality, such as the following result of Sadikova \cite{SAD}. Let $F$ and $G$ be two cumulative distribution functions on $\mathbb{R}^2$ with corresponding characteristic functions $\phi$ and $\varphi$, and assume that
	\[ C_1 = \sup_{(x,y) \in \mathbb{R}^2} \frac{\partial G (x,y)}{\partial x} < \infty \quad \textrm{and} \quad C_2 = \sup_{(x,y) \in \mathbb{R}^2} \frac{\partial G (x,y)}{\partial y} < \infty . \]
	Then for any $T>0$,
	\begin{equation}\label{sadikova}
		\begin{split} \sup_{(x,y) \in \mathbb{R}^2} &|F(x,y) - G(x,y)| \\ &\le \frac{1}{2 \pi^2} \int_{-T}^T \int_{-T}^T \left| \frac{\phi (t_1, t_2) - \phi (t_1,0) \phi (0,t_2) - \varphi (t_1,t_2) + \varphi (t_1,0) \varphi (0,t_2)}{t_1 t_2} \right| \, \mathrm{d}t_1 \mathrm{d}t_2 \\ &\phantom{\le{}}+2 \sup_{x \in \mathbb{R}} |F(x,\infty) - G(x,\infty)| + 2 \sup_{y \in \mathbb{R}} |F(\infty,y) - G(\infty,y)| + (6 \sqrt{2} + 8 \sqrt{3}) \frac{C_1+C_2}{T} . \end{split}
	\end{equation}
	
	\begin{lem}\label{sumRaklemma} For any even integer $K \ge 4$, let
		\[ X_K = \frac{1}{\frac{\pi^3}{24 \log 2}K} \Bigg( \sum_{\substack{k=1 \\ k \textrm{ odd}}}^K R(a_k) - A_K \Bigg) \quad \textrm{and} \quad Y_K = \frac{1}{\frac{\pi^3}{24 \log 2}K} \Bigg( \sum_{\substack{k=1 \\ k \textrm{ even}}}^K R(a_k) - A_K \Bigg) \]
		with $A_K = \frac{\pi^2}{12\log 2} K \log K - \kappa'' K$ and the constant
		\begin{equation}\label{kappa''}
			\kappa'' = \frac{\pi^2}{12 \log 2} \left( \gamma + \log \frac{4 \log 2}{\pi} \right) + \frac{1}{2 \log 2} \sum_{m=1}^{\infty} \left( \frac{\log (2m^2 +1)}{m^2} - 2 \log \frac{2m^2 +2}{2m^2 +1} \right) .
		\end{equation}
		Then $(X_K, Y_K) \overset{d}{\to} \mathrm{Stab}(1,1) \otimes \mathrm{Stab}(1,1)$ with rate $\ll K^{-1/2} \log K$ in the Kolmogorov metric, and
		\begin{equation}\label{XkplusminusYk}
			\frac{X_K+Y_K}{2} - \frac{2 \log 2}{\pi} \overset{d}{\to} \mathrm{Stab}(1,1) \quad \textrm{and} \quad \frac{X_K-Y_K}{2} \overset{d}{\to} \mathrm{Cauchy}
		\end{equation}
		with rate $\ll K^{-1/2}$ in the Kolmogorov metric.
	\end{lem}
	
	\begin{proof} Let $\phi (t_1,t_2) = \mathbb{E}_{\mu} e^{i t_1 X_K + i t_2 Y_K}$ be the characteristic function of the random vector $(X_K, Y_K)$, and let $\varphi (t) = \exp (-|t| (1+\frac{2i}{\pi} \mathrm{sgn} (t) \log |t|))$ be the characteristic function of $\mathrm{Stab}(1,1)$. We start by estimating $\phi (t_1, t_2)$ in the range $|t_1|, |t_2| \le K/(\log K)^3$.
		
		The random variables
		\[ Z_k = Z_{k,K} = \frac{t_1}{\frac{\pi^3}{24 \log 2}K} R(a_{2k-1}) +\frac{t_2}{\frac{\pi^3}{24 \log 2}K} R(a_{2k}), \quad 1 \le k \le K/2 \]
		are $\psi$-mixing with exponential rate. Observe that with a suitable constant $C>0$,
		\[ \begin{split} \mathbb{E}_{\mu} (1-\cos Z_k) &\ge \mu \left( \frac{\pi}{2} \le |Z_k| \le \frac{3 \pi}{2} \right) \\ &\gg \mu_{\mathrm{Gauss}} \left( \frac{CK}{|t_1|} \le a_1 \le \frac{2CK}{|t_1|}, \,\, a_2=1 \right) + \mu_{\mathrm{Gauss}} \left( a_1=1, \,\, \frac{CK}{|t_2|} \le a_2 \le \frac{2CK}{|t_2|} \right) \\ &\gg \frac{|t_1| + |t_2|}{K} . \end{split} \]
		Using $R(x) \ll x$ and $\mu (a_k=n) \ll \mu_{\mathrm{Gauss}} (a_1 =n) \ll 1/n^2$, we also have
		\[ \begin{split} \mathbb{E}_{\mu} |e^{iZ_k} -1| &\ll \mathbb{E}_{\mu} \min \left\{ \frac{|t_1|}{K} R(a_{2k-1}), 1  \right\} + \mathbb{E}_{\mu} \min \left\{ \frac{|t_2|}{K} R(a_{2k}), 1  \right\} \\ &\ll \sum_{1 \le n \le K/|t_1|} \frac{|t_1|}{K}n \cdot \frac{1}{n^2} + \sum_{n>K/|t_1|} \frac{1}{n^2} + \sum_{1 \le n \le K/|t_2|} \frac{|t_2|}{K}n \cdot \frac{1}{n^2} + \sum_{n>K/|t_2|} \frac{1}{n^2} \\ &\ll \frac{|t_1|}{K} \log \frac{K}{|t_1|} + \frac{|t_2|}{K} \log \frac{K}{|t_2|} . \end{split} \]
		An application of \cite[Lemma 1]{HE} with $m=P \approx (\frac{|t_1|}{K} \log \frac{K}{|t_1|} + \frac{|t_2|}{K} \log \frac{K}{|t_2|})^{-1/2} \gg (\log K)^{3/2} (\log \log K)^{-1/2}$ thus gives
		\begin{equation}\label{Zk}
			\begin{split} \Bigg| \mathbb{E}_{\mu} \exp &\left( i \sum_{k=1}^{K/2} Z_k \right) - \exp \left( \sum_{k=1}^{K/2} \mathbb{E}_{\mu} (e^{iZ_k} -1) \right) \Bigg| \\ &\ll K \left( \frac{|t_1|}{K} \log \frac{K}{|t_1|} + \frac{|t_2|}{K} \log \frac{K}{|t_2|} \right)^2 e^{-a(|t_1|+|t_2|)} + e^{-am} K \left( \frac{|t_1|}{K} \log \frac{K}{|t_1|} + \frac{|t_2|}{K} \log \frac{K}{|t_2|} \right) \\ &\ll \frac{|t_1|^{3/2} + |t_2|^{3/2}}{K^{1/2}}e^{-a(|t_1|+|t_2|)} +\frac{|t_1| (1+ |\log |t_1||) + |t_2| (1+ |\log |t_2||)}{ K^{100}}. \end{split}
		\end{equation}
		In the second step we used the fact that the assumption $|t_1|, |t_2| \le K/(\log K)^3$ ensures that
		\[ \left( \frac{|t_1|}{K} \log \frac{K}{|t_1|} + \frac{|t_2|}{K} \log \frac{K}{|t_2|}  \right)^2 \ll \frac{|t_1|^{3/2} + |t_2|^{3/2}}{K^{3/2}} . \]
		Since $Z_k$ is $\mathcal{A}_{2k-1}^{\infty}$-measurable, by the Gauss--Kuzmin theorem \eqref{gausskuzmin},
		\[ \left| \mathbb{E}_{\mu} (e^{i Z_k}-1) - \mathbb{E}_{\mu_{\mathrm{Gauss}}} (e^{i Z_k}-1) \right| \ll e^{-ak} \mathbb{E}_{\mu} |e^{iZ_k}-1| \ll e^{-ak} \left( \frac{|t_1|}{K} \log \frac{K}{|t_1|} + \frac{|t_2|}{K} \log \frac{K}{|t_2|} \right) . \]
		The variables $Z_k$, $1 \le k \le K/2$ are identically distributed under $\mu_{\mathrm{Gauss}}$, and thus we can apply Lemma \ref{gausstransformationlemma} to deduce
		\[ \begin{split} \mathbb{E}_{\mu_{\mathrm{Gauss}}} (e^{i Z_k}-1) &= -\frac{2 |t_1|}{K} - \frac{4i}{\pi} \cdot \frac{t_1 \log |t_1|}{K} + i \frac{A_K}{\frac{\pi^3}{48 \log 2}K^2} t_1 -\frac{2 |t_2|}{K} - \frac{4i}{\pi} \cdot \frac{t_2 \log |t_2|}{K} + i \frac{A_K}{\frac{\pi^3}{48 \log 2}K^2} t_2 \\ &\phantom{={}}+O \left( \frac{|t_1|^{3/2} + |t_2|^{3/2}}{K^{3/2}} \right) . \end{split} \]
		By the previous two formulas,
		\[ \begin{split} \exp \left( \sum_{k=1}^{K/2} \mathbb{E}_{\mu} (e^{iZ_k} -1) \right) &= \varphi (t_1) \varphi (t_2) \exp \left( i \frac{A_K}{\frac{\pi^3}{24 \log 2}K} (t_1+t_2) \right) \\ &\phantom{={}} \times \left( 1+ O \left( \frac{|t_1|^{3/2} + |t_2|^{3/2}}{K^{1/2}} + \frac{|t_1|}{K} \log \frac{K}{|t_1|} + \frac{|t_2|}{K} \log \frac{K}{|t_2|} \right) \right) . \end{split} \]
		The previous formula together with \eqref{Zk} and the fact that $|\varphi (t_1) \varphi (t_2)| = \exp (-(|t_1|+|t_2|))$ leads to the estimate
		\begin{equation}\label{charfunctionestimate}
			\begin{split} \left| \phi (t_1, t_2) -\varphi (t_1) \varphi (t_2) \right| &\ll \left( \frac{|t_1|^{3/2} + |t_2|^{3/2}}{K^{1/2}} + \frac{|t_1|}{K} \log \frac{K}{|t_1|} + \frac{|t_2|}{K} \log \frac{K}{|t_2|} \right) e^{-a(|t_1|+|t_2|)} \\ &\phantom{\ll{}}+ \frac{|t_1| (1+ |\log |t_1||) + |t_2| (1+ |\log |t_2||)}{K^{100}} \end{split}
		\end{equation}
		uniformly in $|t_1|, |t_2| \le K/(\log K)^3$.
		
		Assuming $|t_1| \le 1/K$, we have
		\[ |e^{it_1 X_K} -1| \le \min \left\{ |t_1 X_K|, 2 \right\} \ll |t_1| \log K + \min \left\{ \frac{|t_1|}{K} \sum_{k=1}^K a_k, 1 \right\} . \]
		An application of Lemma \ref{sumaktailslemma} yields
		\[ \mathbb{E}_{\mu} |e^{it_1 X_K}-1| \ll |t_1| \log K + \int_0^1 \mu \left( \frac{|t_1|}{K} \sum_{k=1}^K a_k \ge x \right) \, \mathrm{d}x \ll |t_1| \log K + \int_{2 |t_1| \log K}^1 \frac{|t_1|}{x} \, \mathrm{d}x \ll |t_1| \log \frac{1}{|t_1|} , \]
		hence
		\begin{equation}\label{charfunctionsmallt1}
			|\phi (t_1, t_2)-\phi (0,t_2)| \le \mathbb{E}_{\mu} |e^{it_1 X_K}-1| \ll |t_1| \log \frac{1}{|t_1|} \quad \textrm{uniformly in } |t_1| \le 1/K .
		\end{equation}
		The same argument shows that
		\begin{equation}\label{charfunctionsmallt2}
			|\phi (t_1, t_2)-\phi (t_1,0)| \le \mathbb{E}_{\mu} |e^{it_2 Y_K}-1| \ll |t_2| \log \frac{1}{|t_2|} \quad \textrm{uniformly in } |t_2| \le 1/K .
		\end{equation}
		
		Finally, assume that $|t_1|, |t_2| \le 1/K$. Then
		\[ |e^{it_1 X_K} -1| \cdot |e^{it_2 Y_K}-1| \ll \left( |t_1| \log K + \min \left\{ \frac{|t_1|}{K} \sum_{k=1}^K a_k, 1 \right\} \right) \left( |t_2| \log K + \min \left\{ \frac{|t_2|}{K} \sum_{k=1}^K a_k, 1 \right\} \right) . \]
		By Lemma \ref{sumaktailslemma}, here
		\[ \begin{split} \mathbb{E}_{\mu} \left( \min \left\{ \frac{|t_1|}{K} \sum_{k=1}^K a_k, 1 \right\} \cdot \min \left\{ \frac{|t_2|}{K} \sum_{k=1}^K a_k, 1 \right\} \right) &\le \int_0^1 \mu \left( \frac{|t_1|}{K} \sum_{k=1}^K a_k \cdot \frac{|t_2|}{K} \sum_{k=1}^K a_k \ge x \right) \, \mathrm{d}x \\ &\ll |t_1 t_2| (\log K)^2 + \int_{4 |t_1 t_2| (\log K)^2}^{1} \frac{|t_1 t_2|^{1/2}}{x^{1/2}} \, \mathrm{d}x \\ &\ll |t_1 t_2|^{1/2}, \end{split} \]
		hence
		\begin{equation}\label{charfunctionsmallt1smallt2}
			|\phi (t_1, t_2) - \phi (t_1,0) - \phi (0,t_2) +1| \le \mathbb{E}_{\mu} \left( |e^{it_1 X_K} -1| \cdot |e^{it_2 Y_K}-1| \right) \ll |t_1 t_2|^{1/2}
		\end{equation}
		uniformly in $|t_1|, |t_2| \le 1/K$.
		
		We are now ready to prove the desired limit laws with rate. Let $F(x,y) = \mu (X_K \le x, \, Y_K \le y)$ denote the cumulative distribution function of $(X_K, Y_K)$, and let $G$ be the cumulative distribution function of $\mathrm{Stab}(1,1)$. In particular, $\mathrm{Stab}(1,1) \otimes \mathrm{Stab}(1,1)$ has characteristic function $\varphi (t_1) \varphi (t_2)$ and cumulative distribution function $G(x) G(y)$.
		
		We first show that $X_K \overset{d}{\to} \mathrm{Stab}(1,1)$ and $Y_K \overset{d}{\to} \mathrm{Stab}(1,1)$ with rate $\ll K^{-1/2}$. Choosing $T=K^{1/2}$ in the classical (one-dimensional) Esseen inequality \cite[p.\ 142]{PE} leads to
		\[ \sup_{x \in \mathbb{R}} |F(x,\infty) - G(x)| \ll \int_{-K^{1/2}}^{K^{1/2}} \left| \frac{\phi (t,0) - \varphi (t)}{t} \right| \, \mathrm{d}t + \frac{1}{K^{1/2}} . \]
		Formula \eqref{charfunctionestimate} with $t_2=0$ gives
		\[ |\phi (t,0) - \varphi (t)| \ll \left( \frac{|t|^{3/2}}{K^{1/2}} + \frac{|t|}{K} \log \frac{K}{|t|} \right) e^{-a |t|} + \frac{|t| (1+|\log |t||)}{K^{100}},  \]
		hence
		\[ \int_{-K^{1/2}}^{K^{1/2}} \left| \frac{\phi (t,0) - \varphi (t)}{t} \right| \, \mathrm{d}t \ll \frac{1}{K^{1/2}}. \]
		Therefore $\sup_{x \in \mathbb{R}} |F(x,\infty) - G(x)| \ll K^{-1/2}$. A similar argument gives $\sup_{y \in \mathbb{R}} |F(\infty, y) - G(y)|\ll K^{-1/2}$.
		
		We now prove that $(X_K,Y_K) \overset{d}{\to} \mathrm{Stab}(1,1) \otimes \mathrm{Stab}(1,1)$ with rate $K^{-1/2} \log K$. The inequality of Sadikova \eqref{sadikova} with $T=K^{1/2}$ together with the one-dimensional rates established above lead to
		\[ \sup_{(x,y) \in \mathbb{R}^2} |F(x,y) - G(x)G(y)| \ll \int_{-K^{1/2}}^{K^{1/2}} \int_{-K^{1/2}}^{K^{1/2}} \left| \frac{\phi (t_1,t_2) - \phi (t_1,0) \phi (0,t_2)}{t_1 t_2} \right| \, \mathrm{d}t_1 \mathrm{d}t_2 + \frac{1}{K^{1/2}} . \]
		Formulas \eqref{charfunctionsmallt1}--\eqref{charfunctionsmallt1smallt2} show that in the range $|t_1|, |t_2| \le 1/K$ we have
		\[ |\phi (t_1, t_2) - \phi(t_1,0) \phi (0,t_2)| \le |\phi(t_1,t_2) - \phi(t_1,0) - \phi(0,t_2)+1| + |\phi(t_1,0) -1| \cdot |1-\phi(0,t_2)| \ll |t_1 t_2|^{1/2}, \]
		hence
		\[ \int_{-1/K}^{1/K} \int_{-1/K}^{1/K} \left| \frac{\phi (t_1,t_2) - \phi (t_1,0) \phi (0,t_2)}{t_1 t_2} \right| \, \mathrm{d}t_1 \mathrm{d}t_2 \ll \frac{1}{K} . \]
		Formula \eqref{charfunctionsmallt1} shows that in the range $|t_1| \le 1/K$ we have
		\[ |\phi (t_1, t_2) - \phi(t_1,0) \phi (0,t_2)| \le |\phi (t_1, t_2) - \phi (0,t_2)| + |\phi (0,t_2)| \cdot |1-\phi (t_1,0)| \ll |t_1| \log \frac{1}{|t_1|}, \]
		hence
		\[ \int_{1/K}^{K^{1/2}} \int_{-1/K}^{1/K} \left| \frac{\phi (t_1,t_2) - \phi (t_1,0) \phi (0,t_2)}{t_1 t_2} \right| \, \mathrm{d}t_1 \mathrm{d}t_2 \ll \frac{(\log K)^2}{K} . \]
		The integral on the rectangles $[-1/K, 1/K] \times [-K^{1/2}, -1/K]$, $[1/K,K^{1/2}] \times [-1/K, 1/K]$ and $[-K^{1/2}, -1/K] \times [-1/K,1/K]$ are also $\ll (\log K)^2 /K$. Finally, formula \eqref{charfunctionestimate} shows that
		\[ \begin{split} |\phi (t_1, t_2) - &\phi (t_1,0) \phi (0,t_2)| \\ &\le |\phi (t_1, t_2) - \varphi (t_1) \varphi (t_2)| + |\phi (t_1, 0)| \cdot |\phi (0,t_2)- \varphi (t_2)| + |\varphi(t_2)| \cdot |\phi (t_1,0) - \varphi (t_1)| \\ &\ll \left( \frac{|t_1|^{3/2} + |t_2|^{3/2}}{K^{1/2}} + \frac{|t_1|}{K} \log \frac{K}{|t_1|} + \frac{|t_2|}{K} \log \frac{K}{|t_2|} \right) e^{-a(|t_1|+|t_2|)} \\ &\phantom{\ll{}}+ \frac{|t_1| (1+|\log |t_1||) + |t_2| (1+|\log |t_2||)}{K^{100}} , \end{split} \]
		hence
		\[ \int_{1/K}^{K^{1/2}} \int_{1/K}^{K^{1/2}} \left| \frac{\phi (t_1,t_2) - \phi (t_1,0) \phi (0,t_2)}{t_1 t_2} \right| \, \mathrm{d}t_1 \mathrm{d}t_2 \ll \frac{\log K}{K^{1/2}}. \]
		The integral on the other three similar squares with $1/K \le |t_1|, |t_2| \le K^{1/2}$ are also $\ll K^{-1/2} \log K$. Therefore $\sup_{(x,y) \in \mathbb{R}^2} |F(x,y) - G(x)G(y)| \ll K^{-1/2} \log K$, as claimed.
		
		Observe that $(X_K+Y_K)/2 - (2 \log 2)/\pi$ has characteristic function $\phi (t/2,t/2)\exp (-it (2 \log 2)/\pi)$, and that $\varphi (t/2)^2 \exp (-it (2 \log 2)/\pi) = \varphi (t)$. Formula \eqref{charfunctionestimate} yields a suitable upper estimate for $|\phi (t/2,t/2)-\varphi(t/2)^2|$, and an application of the one-dimensional Esseen inequality shows that $(X_K+Y_K)/2 - (2 \log 2)/\pi \overset{d}{\to} \mathrm{Stab}(1,1)$ with rate $K^{-1/2}$, as claimed.
		
		Similarly, $(X_K - Y_K)/2$ has characteristic function $\phi (t/2, -t/2)$, and $\varphi(t/2)\varphi(-t/2) = \exp (-|t|)$ is the characteristic function of the standard Cauchy distribution. Formula \eqref{charfunctionestimate} yields a suitable upper estimate for $|\phi (t/2,-t/2)-\varphi(t/2) \varphi (-t/2)|$, and an application of the one-dimensional Esseen inequality shows that $(X_K-Y_K)/2 \overset{d}{\to} \mathrm{Cauchy}$ with rate $K^{-1/2}$, as claimed.
	\end{proof}
	
	\subsection{Completing the proof}\label{completesection}
	
	In this section, we compute the contribution of a function of bounded variation $g$, and then give the proof of Theorem \ref{maintheorem}.
	\begin{lem}\label{bvlemma} Let $g$ be a $1$-periodic function that is of bounded variation on $[0,1]$. Then
		\[ \mu \left( \left| \frac{1}{\log N} \sum_{n=1}^N \frac{g(n \alpha)}{n} - \int_0^1 g(x) \, \mathrm{d}x \right| \ge \frac{1}{(\log N)^{1/2}} \right) \ll \frac{1}{(\log N)^{1/2}} \]
		with an implied constant depending only on $L$, $A$ and $g$.
	\end{lem}
	
	\begin{proof} Let
		\[ D_n(\alpha) = \sup_{I \subseteq [0,1]} \left| \sum_{\ell=1}^n \mathds{1}_{I}(\{ \ell \alpha \}) - n \lambda (I) \right| \]
		denote the discrepancy of the point set $\{ \ell \alpha \}$, $1 \le \ell \le n$, where the supremum is over all intervals $I \subseteq [0,1]$. Koksma's inequality \cite[p.\ 143]{KN} states that $|\sum_{\ell=1}^n g(\ell \alpha) - n \int_0^1 g(x) \, \mathrm{d}x |\le V(g;[0,1]) D_n(\alpha)$, where $V(g;[0,1])$ denotes the total variation of $g$ on $[0,1]$. Summation by parts yields
		\[ \sum_{n=1}^N \frac{g(n \alpha)}{n} = \sum_{n=1}^{N-1} \frac{\sum_{\ell=1}^n g(\ell \alpha)}{n(n+1)} + \frac{1}{N} \sum_{\ell=1}^N g(\ell \alpha) , \]
		thus
		\[ \left| \sum_{n=1}^N \frac{g(n \alpha)}{n} - \log N \int_0^1 g(x) \, \mathrm{d}x \right| \ll \sum_{n=1}^{N-1} \frac{D_n(\alpha)}{n^2} + \frac{D_N(\alpha)}{N} . \]
		A classical discrepancy estimate \cite[p.\ 126]{KN} states that for any $q_k \le n < q_{k+1}$, we have $D_n(\alpha) \le 2(a_1 + \cdots + a_{k+1})$. Hence
		\[ \left| \sum_{n=1}^N \frac{g(n \alpha)}{n} - \log N \int_0^1 g(x) \, \mathrm{d}x \right| \ll \sum_{k=0}^{K_N^*} \frac{a_1+\cdots +a_{k+1}}{q_k} , \]
		where $K_N^*=K_N^*(\alpha)$ is the random index for which $q_{K_N^*} \le N < q_{K_N^*+1}$. Since $\frac{12 \log 2}{\pi^2}<1$, Lemma \ref{KNlemma} shows that $K_N^*+1 \le \log N$ outside a set of $\mu$-measure $\ll (\log N)^{-1/2}$. Since $q_k \ge F_{k+1}$ for all $k \ge 0$, with $F_1=F_2=1$, $F_k=F_{k-1} + F_{k-2}$ denoting the sequence of Fibonacci numbers,
		\[ \mu \left( \left| \sum_{n=1}^N \frac{g(n \alpha)}{n} - \log N \int_0^1 g(x) \, \mathrm{d}x \right| \gg \sum_{1 \le k \le \log N} \frac{a_1+\cdots +a_k}{F_k} \right) \ll \frac{1}{(\log N)^{1/2}} . \]
		As $\sum_{k=1}^{\infty} F_k^{-1/2} < \infty$, the inequality $\sum_{1 \le k \le \log N} (a_1+\cdots +a_k)/F_k \gg (\log N)^{1/2}$ implies that $(a_1+\cdots +a_k)/F_k^{1/2} \gg (\log N)^{1/2}$ for some $k$. An application of Lemma \ref{sumaktailslemma} thus gives
		\[ \begin{split} \mu \left( \sum_{1 \le k \le \log N} \frac{a_1+\cdots +a_k}{F_k} \gg (\log N)^{1/2} \right) &\le \sum_{1 \le k \le \log N} \mu \left( \frac{a_1+\cdots +a_k}{F_k^{1/2}} \gg (\log N)^{1/2} \right) \\ &\ll \sum_{1 \le k \le \log N} \frac{k}{F_k^{1/2} (\log N)^{1/2}} \\ &\ll \frac{1}{(\log N)^{1/2}} , \end{split} \]
		and the claim follows.
	\end{proof}
	
	\begin{proof}[Proof of Theorem \ref{maintheorem}] Lemma \ref{lipschitzlemma} shows that it is enough to prove the desired limit law with rate under the assumption \eqref{lipschitz}. Lemma \ref{bvlemma} reduces the theorem to the special case $g_1=g_2=0$.
		
		Recall that $K_N$ is an even integer such that $|K_N - \frac{12 \log 2}{\pi^2} \log N| \le 1$. By Lemmas \ref{Ruklemma}, \ref{uktoaklemma} and \ref{sumRaklemma}, we have
		\[ \frac{1}{\sigma_N} \Bigg( \sum_{\substack{n=1 \\ 0< \langle n \alpha \rangle < 1/(2n)}}^N \frac{1}{n \langle n \alpha \rangle} - E_N, \sum_{\substack{n=1 \\ -1/(2n) < \langle n \alpha \rangle < 0}}^N \frac{1}{n |\langle n \alpha \rangle|} - E_N \Bigg) \overset{d}{\to} \mathrm{Stab}(1,1) \otimes \mathrm{Stab}(1,1) \]
		with rate $\ll (\log N)^{-1/4} (\log \log N)^{1/4}$ in the Kolmogorov metric, where $\sigma_N = \frac{\pi^3}{24 \log 2} K_N = \frac{\pi}{2} \log N + O(1)$ and
		\[ E_N = \frac{\kappa}{2} K_N + \frac{\pi^2}{12 \log 2} K_N \log K_N - \kappa'' K_N = \log N \log \log N - c\log N +O(\log \log N) \]
		with the constant
		\[ c = - \frac{6 \log 2}{\pi^2} \kappa - \log \frac{12 \log 2}{\pi^2} + \frac{12 \log 2}{\pi^2} \kappa'' = \gamma + \log \frac{2 \pi}{3} + \frac{12}{\pi^2} \sum_{m=1}^{\infty} \frac{\log m}{m^2} -1 . \]
		The remaining infinite series can be evaluated using the Dirichlet series $\zeta'(s) = - \sum_{m=1}^{\infty} m^{-s} \log m$, $\mathrm{Re} \, s >1$. Lemma \ref{awaylemma} yields the contribution of the terms with $\langle n \alpha \rangle \ge 1/(2n)$ resp.\ $\langle n \alpha \rangle \le -1/(2n)$, and the claim follows.
	\end{proof}
	
	The proof of Corollary \ref{maincorollary} is identical, based on the one-dimensional limit laws \eqref{XkplusminusYk} with rate from Lemma \ref{sumRaklemma}.
	
	\section{Proof of Theorem \ref{p>1theorem}}\label{p>1section}
	
	Throughout this section, $\alpha \sim \mu$ with some $\mu \ll \lambda$ that satisfies \eqref{lipschitz}, and $p>0$. All implied constants depend only on $L$, $A$ and $p$.
	
	We start with a generalization of Lemma \ref{sumaktailslemma} to power sums of partial quotients.
	\begin{lem}\label{sumakptailslemma} If $p>1$, then for any integers $M \ge 0$ and $K \ge 1$ and any real $t>0$,
		\[ \mu \left( \sum_{k=M+1}^{M+K} a_k^p \ge t K^p \right) \ll t^{-1/p} . \]
		If $1/2 \le p<1$, then for any integers $M \ge 0$ and $K \ge 1$ and any real $t>0$,
		\[ \mu \left( \left| \sum_{k=M+1}^{M+K} a_k^p - m_p K \right| \ge t K^p \right) \ll \left\{ \begin{array}{ll} t^{-1/p} & \textrm{if } 1/2<p<1, \\ t^{-2}(1+|\log (tK)|) & \textrm{if } p=1/2, \end{array} \right. \]
		where $m_p=\frac{1}{\log 2} \sum_{n=1}^{\infty} n^p \log \left( 1+\frac{1}{n(n+2)} \right)$.
	\end{lem}
	
	\begin{proof} We follow the approach of Diamond and Vaaler \cite{DV}. We may assume that $t$ is greater than any prescribed constant depending on $p$.
		
		An application of the union bound shows that
		\[ \mu_{\mathrm{Gauss}} \left( \max_{M+1 \le k \le M+K} a_k \ge t^{1/p} K \right) \le K \mu_{\mathrm{Gauss}} (a_1 \ge t^{1/p}K) \ll t^{-1/p} . \]
		In particular,
		\begin{equation}\label{indicator}
			\mu_{\mathrm{Gauss}} \left( \sum_{k=M+1}^{M+K} a_k^p \neq \sum_{k=M+1}^{M+K} a_k^p \mathds{1}_{\{ a_k \le t^{1/p} K \}} \right) \ll t^{-1/p} .
		\end{equation}
		
		Assume first that $p>1$. By exponentially fast mixing (cf.\ $\rho$-mixing \cite{BR}), the random variables $X_k=a_k^p \mathds{1}_{\{ a_k \le t^{1/p} K \}}$ satisfy
		\[ \mathrm{Cov}_{\mu_{\mathrm{Gauss}}} \left( X_k, X_{\ell} \right) \ll e^{-a|k-\ell|} (\mathrm{Var}_{\mu_{\mathrm{Gauss}}} X_k)^{1/2} (\mathrm{Var}_{\mu_{\mathrm{Gauss}}} X_{\ell})^{1/2} \]
		with some constant $a>0$. Here
		\[ \mathrm{Var}_{\mu_{\mathrm{Gauss}}} X_k \le \frac{1}{\log 2} \sum_{1 \le n \le t^{1/p} K} n^{2p} \log \left( 1+\frac{1}{n(n+2)} \right) \ll t^{2-1/p} K^{2p-1}, \]
		therefore
		\[ \mathrm{Var}_{\mu_{\mathrm{Gauss}}} \left( \sum_{k=M+1}^{M+K} X_k \right) \ll \sum_{k, \ell =M+1}^{M+K} e^{-a|k-\ell|} t^{2-1/p} K^{2p-1} \ll t^{2-1/p} K^{2p} . \]
		Further,
		\[ \mathbb{E}_{\mu_{\mathrm{Gauss}}} \sum_{k=M+1}^{M+K} X_k = \frac{K}{\log 2} \sum_{1 \le n \le t^{1/p}K} n^p \log \left( 1+\frac{1}{n(n+2)} \right) \ll t^{1-1/p} K^p , \]
		which is less than $tK^p/2$ provided that $t$ is large enough. The Chebyshev inequality thus yields
		\[ \mu_{\mathrm{Gauss}} \left( \sum_{k=M+1}^{M+K} X_k \ge t K^p \right) \le \mu_{\mathrm{Gauss}} \left( \left| \sum_{k=M+1}^{M+K} (X_k - \mathbb{E}_{\mu_{\mathrm{Gauss}}} X_k ) \right| \ge \frac{t K^p}{2} \right) \ll t^{-1/p} , \]
		which together with \eqref{indicator} proves the claim for the Gauss measure.
		
		Assume next that $1/2 \le p<1$. Then
		\[ \mathrm{Var}_{\mu_{\mathrm{Gauss}}} X_k \ll \left\{ \begin{array}{ll} t^{2-1/p} K^{2p-1} & \textrm{if } 1/2<p<1, \\ \log (tK) & \textrm{if } p=1/2, \end{array} \right. \]
		and
		\[ \mathbb{E}_{\mu_{\mathrm{Gauss}}} \sum_{k=M+1}^{M+K} X_k = m_p K + O(t^{1-1/p} K^p) , \]
		where the error term is smaller than $tK^p/2$ provided that $t$ is large enough. We conclude similarly by the Chebyshev inequality.
		
		The density $\mathrm{d}\mu / \mathrm{d}\mu_{\mathrm{Gauss}}$ is bounded above by a constant depending only on $L$, hence the tail estimates in the claim remain true for any $\mu$ that satisfies \eqref{lipschitz}.
	\end{proof}
	
	The following lemma improves \cite[Corollary 3.2]{BD}.
	\begin{lem}\label{p>1fourierlemma} Let $p>1$. For any $t \in \mathbb{R}$ such that $|t| \le 1/2$,
		\[ \mathbb{E}_{\mu_{\mathrm{Gauss}}} (e^{it a_1^p}-1) = - \frac{1}{\log 2} \cos \left( \frac{\pi}{2p} \right) \Gamma \left( 1-\frac{1}{p} \right) |t|^{1/p} \left( 1-i \, \mathrm{sgn}(t) \tan \frac{\pi}{2p} \right) - H(t) +O(\varepsilon(t)) , \]
		where
		\[ H(t)= \left\{ \begin{array}{ll} i d_p t & \textrm{if } 1<p<2, \\ \frac{i}{\log 2} t \log \frac{1}{|t|} & \textrm{if } p=2, \\ 0 & \textrm{if } p>2, \end{array} \right. \]
		with the constant
		\begin{equation}\label{dpdefinition}
			d_p = \frac{1}{(p-1) \log 2} + \frac{1}{\log 2} \int_1^{\infty} \frac{x^p-\lfloor x \rfloor^p + x^{p-1}}{x^2+x} \, \mathrm{d}x, \quad 1<p<2,
		\end{equation}
		and
		\[ \varepsilon(t) = \left\{ \begin{array}{ll} |t|^{2/p} & \textrm{if } 1<p \le 2, \\ |t|^{1/(p-1)} & \textrm{if } p>2. \end{array} \right. \]
	\end{lem}
	
	\begin{proof} It is enough to prove the claim for $0<t \le 1/2$, as the case of negative $t$ follows from complex conjugation. The expected value in the claim can be expressed as
		\[ \mathbb{E}_{\mu_{\mathrm{Gauss}}} (e^{i t a_1^p} -1) = \int_0^1 (e^{it \lfloor 1/x \rfloor^p}-1) \, \mathrm{d} \mu_{\mathrm{Gauss}} (x) = \frac{1}{\log 2} \int_1^{\infty} \frac{e^{it \lfloor x \rfloor^p}-1}{x^2+x} \, \mathrm{d}x = \frac{1}{\log 2} (I_1 + I_2 + I_3), \]
		where
		\[ I_1= \int_1^{\infty} \frac{e^{itx^p}-1}{x^2} \, \mathrm{d}x, \quad I_2= \int_1^{\infty} (e^{itx^p} -1) \left( \frac{1}{x^2+x} - \frac{1}{x^2} \right) \, \mathrm{d}x , \quad I_3 = \int_1^{\infty} \frac{e^{it \lfloor x \rfloor^p} - e^{itx^p}}{x^2+x} \, \mathrm{d}x. \]
		A substitution gives $I_1 = t^{1/p} \int_{t^{1/p}}^{\infty} (e^{ix^p}-1)/x^2 \, \mathrm{d}x$, and here the integral extended to $[0,\infty)$ evaluates to
		\[ \int_0^{\infty} \frac{e^{ix^p}-1}{x^2} \, \mathrm{d}x = - \cos \left( \frac{\pi}{2p} \right) \Gamma \left( 1-\frac{1}{p} \right) + i \sin \left( \frac{\pi}{2p} \right) \Gamma \left( 1-\frac{1}{p} \right) . \]
		In particular,
		\[ I_1 = - \cos \left( \frac{\pi}{2p} \right) \Gamma \left( 1-\frac{1}{p} \right) t^{1/p} \left( 1-i \tan \frac{\pi}{2p} \right) + I_1' \]
		with $I_1' = - t^{1/p} \int_0^{t^{1/p}} (e^{ix^p}-1)/x^2 \, \mathrm{d}x$, and it remains to show that
		\begin{equation}\label{I1I2I3}
			\frac{1}{\log 2} (I_1' + I_2 + I_3) = - H(t) + O(\varepsilon (t)).
		\end{equation}
		
		Assume first that $1<p<2$. Then
		\[\begin{split} I_1' &= - t^{1/p} \int_0^{t^{1/p}} \frac{ix^p}{x^2} \, \mathrm{d}x - t^{1/p} \int_0^{t^{1/p}} \frac{e^{ix^p}-1-ix^p}{x^2} \, \mathrm{d}x = - \frac{i}{p-1} t + O \left( t^{1/p} \int_0^{t^{1/p}} x^{2p-2} \, \mathrm{d}x \right) \\ &= - \frac{i}{p-1} t + O(t^2) , \end{split} \]
		as well as
		\[ \begin{split} I_2 &= \int_1^{\infty} \frac{-itx^p}{x^3+x^2} \, \mathrm{d}x + \int_1^{\infty} \frac{1+itx^p - e^{itx^p}}{x^3+x^2} \, \mathrm{d}x \\ &= - it \int_1^{\infty} \frac{x^{p-1}}{x^2 +x} \, \mathrm{d}x + O \left( \int_1^{\infty} \frac{\min \{ t^2 x^{2p}, 1+ tx^p \}}{x^3} \, \mathrm{d}x \right) \\ &= - it \int_1^{\infty} \frac{x^{p-1}}{x^2 +x} \, \mathrm{d}x + O (t^{2/p}) , \end{split}  \]
		and
		\[ \begin{split} I_3 &= \int_1^{\infty} \frac{it (\lfloor x \rfloor^p -x^p)}{x^2+x} \, \mathrm{d}x + \int_1^{\infty} \frac{(e^{itx^p}-1)it (\lfloor x \rfloor^p -x^p)}{x^2+x} \, \mathrm{d}x \\ &\phantom{={}} + \int_1^{\infty} \frac{e^{itx^p} \left( e^{it(\lfloor x \rfloor^p-x^p)} -1 - it (\lfloor x \rfloor^p -x^p) \right)}{x^2+x} \, \mathrm{d}x \\ &= it \int_1^{\infty} \frac{\lfloor x \rfloor^p - x^p}{x^2+x} \, \mathrm{d}x + O \left( \int_1^{\infty} \frac{\min \{ tx^p, 1 \} tx^{p-1}}{x^2} \, \mathrm{d}x + \int_1^{\infty} \frac{\min \{ t^2 x^{2p-2}, 1+t x^{p-1} \}}{x^2} \, \mathrm{d}x \right) \\ &= it \int_1^{\infty} \frac{\lfloor x \rfloor^p - x^p}{x^2+x} \, \mathrm{d}x + O(t^{2/p}) . \end{split} \]
		The previous three formulas yield \eqref{I1I2I3}.
		
		Assume next that $p = 2$. Then
		\[ |I_1'| \le t^{1/2} \int_0^{t^{1/2}} \frac{|e^{ix^2}-1|}{x^2} \, \mathrm{d}x \le t, \]
		as well as
		\[ \begin{split} I_2 &= \int_1^{t^{-1/2}} \frac{-it x^2}{x^3+x^2} \, \mathrm{d}x + \int_1^{t^{-1/2}} \frac{1+itx^2 - e^{itx^2}}{x^3+x^2} \, \mathrm{d}x + \int_{t^{-1/2}}^{\infty} \frac{1-e^{itx^2}}{x^3+x^2} \, \mathrm{d}x \\ &= -it \log \frac{t^{-1/2}+1}{2} + O \left( \int_1^{t^{-1/2}} \frac{t^2 x^4}{x^3} \, \mathrm{d}x + \int_{t^{-1/2}}^{\infty} \frac{1}{x^3} \, \mathrm{d}x \right) \\ &= - \frac{it}{2} \log \frac{1}{t} +O(t) , \end{split} \]
		and
		\[ \begin{split} I_3 &= \int_1^{1/t} \frac{e^{itx^2} it (\lfloor x \rfloor^2 - x^2)}{x^2 +x} \, \mathrm{d}x + \int_1^{1/t} \frac{e^{itx^2} (e^{it (\lfloor x \rfloor^2 - x^2)} - 1 - it (\lfloor x \rfloor^2 - x^2))}{x^2+x} \, \mathrm{d}x +\int_{1/t}^{\infty} \frac{e^{it\lfloor x \rfloor^2} - e^{itx^2}}{x^2+x} \, \mathrm{d}x \\ &= -it \int_1^{1/t} \frac{e^{itx^2} \{ x \} (\lfloor x \rfloor +x)}{x^2+x} \, \mathrm{d}x + O \left( \int_1^{1/t} \frac{t^2 x^2}{x^2} \, \mathrm{d}x + \int_{1/t}^{\infty} \frac{1}{x^2} \, \mathrm{d}x \right) \\ &= -2it \int_1^{1/t} \frac{e^{itx^2} \{ x\}}{x} \, \mathrm{d}x +O(t). \end{split} \]
		In the last step we replaced $\lfloor x \rfloor +x$ by $2x$ in the numerator, and $x^2+x$ by $x^2$ in the denominator. Here the average value $1/2$ of $\{ x \}$ would give
		\[ -2it \int_1^{1/t} \frac{e^{itx^2} (1/2)}{x} \, \mathrm{d}x = -\frac{it}{2} \left( \int_t^1 \frac{1}{y} \, \mathrm{d}y + \int_t^1 \frac{e^{iy}-1}{y} \, \mathrm{d}y + \int_1^{1/t} \frac{e^{iy}}{y} \, \mathrm{d}y \right) = - \frac{it}{2} \log \frac{1}{t} +O(t), \]
		where the last integral was estimated using integration by parts. On the other hand, $1/x = 1/n + O(1/n^2)$ and $e^{itx^2} = e^{itn^2} +O(tn)$ on the interval $x\in [n,n+1]$, hence
		\[ \int_n^{n+1} \frac{e^{itx^2} (\{ x \} -1/2)}{x} \, \mathrm{d}x = \frac{e^{itn^2}}{n} \int_n^{n+1} \left( \{ x \} - \frac{1}{2} \right) \, \mathrm{d}x + O \left( \frac{1}{n^2} +t \right) = O \left( \frac{1}{n^2} +t \right) , \]
		and by summing over $1 \le n \le 1/t$, we obtain
		\[ -2it \int_1^{1/t} \frac{e^{itx^2} (\{ x \} - 1/2)}{x} \, \mathrm{d}x = O(t) . \]
		Therefore $I_3 = -\frac{it}{2} \log \frac{1}{t}+O(t)$, and \eqref{I1I2I3} follows.
		
		Finally, assume that $p>2$. Then
		\[ \begin{split} |I_1'| &\le t^{1/p} \int_0^{t^{1/p}} x^{p-2} \, \mathrm{d}x \ll t, \\ |I_2| &\le \int_1^{\infty} \frac{\min \{ tx^p, 2 \}}{x^3} \, \mathrm{d}x \ll t^{2/p}, \\ |I_3| &\ll \int_1^{\infty} \frac{\min \{ tx^{p-1}, 1 \}}{x^2} \, \mathrm{d}x \ll t^{1/(p-1)}, \end{split} \]
		and \eqref{I1I2I3} follows.
	\end{proof}
	
	We now prove the analogue of the limit law \eqref{sumaklimitlaw} for power sums. A limit law itself for power sums $\sum_{k=1}^K a_k^p$ is classical \cite[p.\ 197--198]{IK}, but to the best of our knowledge the rate of convergence has not explicitly appeared anywhere in the literature. Note that the centering term $A_K$ is not necessary for the limit law itself, only for the rate.
	\begin{lem}\label{sumakplimitlawlemma} Let $p>1$. For any even integer $K \ge 2$, let
		\[ X_K = \frac{1}{\sigma K^p} \Bigg( \sum_{\substack{k=1 \\ k \textrm{ odd}}}^K a_k^p + A_K \Bigg) \quad \textrm{and} \quad Y_K = \frac{1}{\sigma K^p}\Bigg( \sum_{\substack{k=1 \\ k \textrm{ even}}}^K a_k^p + A_K \Bigg) , \]
		where
		\[ A_K = \left\{ \begin{array}{ll} \frac{d_p}{2} K & \textrm{if } 1<p<2, \\ \frac{1}{\log 2} K \log K & \textrm{if } p=2, \\ 0 & \textrm{if } p>2, \end{array} \right. \]
		$d_p$ is as in \eqref{dpdefinition}, and $\sigma = \left( \frac{1}{2\log 2} \cos \left( \frac{\pi}{2p} \right) \Gamma \left( 1-\frac{1}{p} \right) \right)^p$. Then $(X_K,Y_K) \overset{d}{\to} \mathrm{Stab}(1/p,1) \otimes \mathrm{Stab} (1/p,1)$ with rate $\ll \varepsilon_K \log K$ in the Kolmogorov metric, and
		\[ \frac{X_K+Y_K}{2^p} \overset{d}{\to} \mathrm{Stab} (1/p,1) \quad \textrm{and} \quad \frac{X_K-Y_K}{2^p} \overset{d}{\to} \mathrm{Stab}(1/p,0) \]
		with rate $\ll \varepsilon_K$ in the Kolmogorov metric, where
		\[ \varepsilon_K= \left\{ \begin{array}{ll} K^{-1} & \textrm{if } 1<p \le 2, \\ K^{-1/(p-1)} & \textrm{if } p>2 . \end{array} \right. \]
	\end{lem}
	
	\begin{proof} Following the steps in the proof of Lemma \ref{gausstransformationlemma} shows that for any $t_1, t_2 \in \mathbb{R}$ such that $|t_1|, |t_2| \le 1/2$,
		\[ \left| \mathbb{E}_{\mu_{\mathrm{Gauss}}} \left( e^{it_1 a_1^p} -1 \right) \left( e^{it_2 a_2^p} -1 \right) \right| \ll |t_1|^{2/p} + |t_2|^{2/p} . \]
		Therefore $\mathbb{E}_{\mu_{\mathrm{Gauss}}} \left( e^{i (t_1 a_1^p + t_2 a_2^p)} -1 \right) = w_p (t_1) + w_p (t_2) + O(|t_1|^{2/p} + |t_2|^{2/p})$ with the same function $w_p(t)=\mathbb{E}_{\mu_{\mathrm{Gauss}}} (e^{it a_1^p}-1)$. An application of Lemma \ref{p>1fourierlemma} thus leads to
		\begin{equation}\label{gausst1t2p}
			\begin{split}
				\mathbb{E}_{\mu_{\mathrm{Gauss}}} \left( e^{i (t_1 a_1^p + t_2 a_2^p)} -1 \right) &= - \frac{\cos \left( \frac{\pi}{2p} \right) \Gamma \left( 1-\frac{1}{p} \right)}{\log 2}|t_1|^{1/p} \left( 1-i \, \mathrm{sgn} (t_1) \tan \frac{\pi}{2p} \right) - H(t_1) \\ &\phantom{={}}- \frac{\cos \left( \frac{\pi}{2p} \right) \Gamma \left( 1-\frac{1}{p} \right)}{\log 2}|t_2|^{1/p} \left( 1-i \, \mathrm{sgn} (t_2) \tan \frac{\pi}{2p} \right) - H(t_2) \\ &\phantom{={}}+ \left\{ \begin{array}{ll} O (|t_1|^{2/p} + |t_2|^{2/p}) & \textrm{if } 1<p \le 2, \\ O\left( |t_1|^{1/(p-1)} + |t_2|^{1/(p-1)} \right) & \textrm{if } p>2 . \end{array} \right. \end{split}
		\end{equation}
		
		Now let $\phi (t_1, t_2) = \mathbb{E}_{\mu} (e^{i t_1 X_K + i t_2 Y_K})$ be the characteristic function of the random vector $(X_K, Y_K)$, and let $\varphi (t) = \exp (-|t|^{1/p} (1-i \, \mathrm{sgn}(t) \tan \frac{\pi}{2p}))$ be the characteristic function of $\mathrm{Stab}(1/p,1)$. We estimate $\phi (t_1,t_2)$ in the range $|t_1|, |t_2| \le K^p / (\log K)^{3p}$. Let
		\[ Z_k = Z_{k,K} = \frac{t_1}{\sigma K^p} a_{2k-1}^p + \frac{t_2}{\sigma K^p} a_{2k}^p, \quad 1 \le k \le K/2 . \]
		Similarly to the proof of Lemma \ref{sumRaklemma}, we deduce
		\[ \mathbb{E}_{\mathrm{\mu}} (1-\cos Z_k) \gg \frac{|t_1|^{1/p} + |t_2|^{1/p}}{K} \quad \textrm{and} \quad \mathbb{E}_{\mu} |e^{iZ_k}-1| \ll \frac{|t_1|^{1/p} + |t_2|^{1/p}}{K} . \]
		An application of \cite[Lemma 1]{HE} with $m=P \approx (\frac{|t_1|^{1/p} + |t_2|^{1/p}}{K})^{-1/2} \gg (\log K)^{3/2}$ now leads to
		\begin{multline*}
			\left| \mathbb{E}_{\mu} \exp \left( i \sum_{k=1}^{K/2} Z_k \right) - \exp \left( \sum_{k=1}^{K/2} \mathbb{E}_{\mu_{\mathrm{Gauss}}} (e^{iZ_k}-1) + O \left( \frac{|t_1|^{1/p} + |t_2|^{1/p}}{K} \right) \right)  \right| \\ \ll \frac{|t_1|^{2/p} + |t_2|^{2/p}}{K} e^{-a(|t_1|^{1/p} + |t_2|^{1/p})} + \frac{|t_1|^{1/p} + |t_2|^{1/p}}{K^{100}}
		\end{multline*}
		with a small constant $a>0$. Together with formula \eqref{gausst1t2p} this shows that
		\[ |\phi (t_1, t_2) - \varphi (t_1) \varphi (t_2)| \ll \left( U(t_1) + U(t_2) \right) e^{-a(|t_1|^{1/p} + |t_2|^{1/p})} +  \frac{|t_1|^{1/p} + |t_2|^{1/p}}{K^{100}} \]
		uniformly in $|t_1|, |t_2| \le K^p / (\log K)^{3p}$, where
		\[ U(t) = \left\{ \begin{array}{ll} \frac{|t|^{1/p} + |t|^{2/p}}{K} & \textrm{if } 1<p<2, \\ \frac{|t|^{1/2}+|t \log |t||}{K} & \textrm{if } p=2, \\ \frac{|t|^{1/p}}{K} +  \frac{|t|^{1/(p-1)}}{K^{1/(p-1)}} & \textrm{if } p>2 . \end{array} \right. \]
		The rest of the proof is identical to that of Lemma \ref{sumRaklemma}.
	\end{proof}
	
	Finally, we estimate the contribution of a $1$-periodic function with a lower order singularity at integers, and then prove Theorem \ref{p>1theorem}.
	\begin{lem}\label{qlemma} Let $g$ be a $1$-periodic function such that $|g(x)| \ll 1/ \| x \|^{p'}$ with some constant $0 \le p' < p$. Then
		\[ \mu \left( \frac{1}{(\log N)^p} \left| \sum_{n=1}^N \frac{g(n \alpha)}{n^p} \right| \ge \frac{1}{(\log N)^{1/2}} \right) \ll \frac{1}{(\log N)^{1/2}} \]
		with an implied constant depending only on $L$, $A$, $g$ and $p$.
	\end{lem}
	
	\begin{proof} We may assume that $1<p' <p$. For any $k \ge 0$, we have
		\[ \sum_{q_k \le n < q_{k+1}} \frac{|g(n\alpha)|}{n^p} \ll \frac{1}{q_k^p} \sum_{1 \le n < q_{k+1}} \frac{1}{\| n \alpha \|^{p'}} . \]
		For any two distinct integers $1 \le n,n' < q_{k+1}$, we have $1 \le |n-n'|<q_{k+1}$, hence by the best approximation property of continued fractions, $\| n \alpha - n' \alpha \| \ge \| q_k \alpha \|$. An application of the pigeonhole principle thus shows that
		\[ \sum_{1 \le n < q_{k+1}} \frac{1}{\| n \alpha \|^{p'}} \le 2 \sum_{j=1}^{\infty} \frac{1}{(j \| q_k \alpha \|)^{p'}} \ll \frac{1}{\| q_k \alpha \|^{p'}} \ll a_{k+1}^{p'} q_k^{p'} , \]
		and consequently
		\[ \left| \sum_{n=1}^N \frac{g(n\alpha)}{n^p} \right| \le \sum_{k=0}^{K_N^*} \sum_{q_k \le n < q_{k+1}} \frac{|g(n\alpha)|}{n^p} \ll \sum_{k=0}^{K_N^*} \frac{a_{k+1}^{p'}}{q_k^{p-p'}} , \]
		where $K_N^*=K_N^*(\alpha)$ is the random index for which $q_{K_N^*} \le N < q_{K_N^* +1}$. Lemma \ref{KNlemma} shows that $K_N^* +1 \le \log N$ outside a set of $\mu$-measure $\ll (\log N)^{-1/2}$. Since $q_k \ge F_{k+1}$, with $F_k$ denoting the sequence of Fibonacci numbers,
		\[ \mu \left( \left| \sum_{n=1}^N \frac{g(n\alpha)}{n^p} \right| \gg \sum_{1 \le k \le \log N} \frac{a_k^{p'}}{F_k^{p-p'}} \right) \ll \frac{1}{(\log N)^{1/2}} . \]
		As $\sum_{k=1}^{\infty} 1/F_k^{(p-p')/2} < \infty$, the inequality $\sum_{1 \le k \le \log N} a_k^{p'} / F_k^{p-p'} \ge (\log N)^{p-1/2}$ implies that $a_k^{p'}/F_k^{(p-p')/2} \gg (\log N)^{p-1/2}$ for some $1 \le k \le \log N$. Therefore
		\[ \begin{split} \mu \left( \sum_{1 \le k \le \log N} \frac{a_k^{p'}}{F_k^{p-p'}} \gg (\log N)^{p-1/2} \right) &\le \sum_{1 \le k \le \log N} \mu \left( \frac{a_k^{p'}}{F_k^{(p-p')/2}} \gg (\log N)^{p-1/2} \right) \\ &\ll \sum_{1 \le k \le \log N} \frac{1}{F_k^{(p-p')/(2p')} (\log N)^{(p-1/2)/p'}} \\ &\ll (\log N)^{- (p-1/2)/p'} . \end{split} \]
		Here $(p-1/2)/p' > 1/2$, and the claim follows.
	\end{proof}
	
	\begin{proof}[Proof of Theorem \ref{p>1theorem}] Lemma \ref{lipschitzlemma} shows that it is enough to prove the desired limit law with rate under the assumption \eqref{lipschitz}. Lemma \ref{qlemma} reduces the theorem to the special case $g_1=g_2=0$. Let $K_N^*$ and $K_N$ be as in Section \ref{closesection}. We give separate proofs in the ranges $1<p<3/2$ and $p \ge 3/2$.
		
		Assume first that $1<p<3/2$. As the proof is very similar to that of Theorem \ref{maintheorem}, we only point out the necessary modifications. We start by deducing an analogue of Lemma \ref{awaylemma}. Consider the sequences of functions
		\[ f_n (x) = \left\{ \begin{array}{ll} 2^p & \text{if } 0 \le x < \frac{1}{2n}, \\ \frac{1}{n^p x^p} & \text{if } \frac{1}{2n} \le x \le \frac{1}{2}, \\ 0 & \text{else}, \end{array} \right. \qquad \textrm{and} \qquad \tilde{f}_n (x) = \left\{ \begin{array}{ll} 2^p & \text{if } 0 \le x < \frac{1}{2n}, \\ 0 & \text{else}. \end{array} \right. \]
		Observe that
		\[ \begin{split} \sum_{n=1}^N \int_0^{\infty} f_n (x) \, \mathrm{d}x &= \frac{2^{p-1}p}{p-1} \log N +O(1), \\  \sum_{n=1}^N \int_0^{\infty} \tilde{f}_n (x) \, \mathrm{d}x &= 2^{p-1} \log N +O(1) , \end{split} \]
		and that for all $1 \le n \le m$,
		\[ \int_0^{\infty} f_n (nx) f_m (mx) \, \mathrm{d}x \ll \frac{1}{m^2} \qquad \textrm{and} \qquad \int_0^{\infty} \tilde{f}_n (nx) \tilde{f}_m (mx) \, \mathrm{d}x \ll \frac{1}{m^2}. \]
		Arguing as in the proof of Lemma \ref{awaylemma}, an application of Lemma \ref{schmidtlemma} leads to
		\begin{equation}\label{pawayfromsingularity}
			\sum_{\substack{n=1 \\ \langle n \alpha \rangle \ge 1/(2n)}}^N \frac{1}{n^p \langle n \alpha \rangle^p} = \frac{2^{p-1}}{p-1} \log N + \xi_{1,N} (\alpha) \quad \textrm{and} \quad \sum_{\substack{n=1 \\ \langle n \alpha \rangle \le - 1/(2n)}}^N \frac{1}{n^p |\langle n \alpha \rangle|^p} = \frac{2^{p-1}}{p-1} \log N + \xi_{2,N} (\alpha)
		\end{equation}
		with some error terms $\xi_{j,N}(\alpha)$, $j=1,2$ that satisfy
		\[ \mu \left( \frac{|\xi_{j,N}(\alpha)|}{(\log N)^p} \ge \frac{(\log \log N)^{1/3}}{(\log N)^{\frac{2p-1}{3}}} \right) \ll \frac{(\log \log N)^{1/3}}{(\log N)^{\frac{2p-1}{3}}} . \]
		
		Let $R_p(x) = \sum_{1 \le j < (x/2)^{1/2}} x^p/j^{2p}$, and $r_{m,p}=\sum_{j=1}^m 1/j^{2p}$ with the convention $r_{0,p}=0$. In particular, $R_p(x)=\zeta(2p) x^p +O(x^{1/2})$. Following the steps in Lemma \ref{Ruklemma} yields
		\[ \sum_{\substack{n=1 \\ 0< \langle n \alpha \rangle < 1/(2n)}}^N \frac{1}{n^p \langle n \alpha \rangle^p} = \sum_{\substack{k=1 \\ k \textrm{ odd}}}^{K_N} R_p (u_k) + \xi_{1,N}(\alpha) \quad \textrm{and} \sum_{\substack{n=1 \\ -1/(2n)< \langle n \alpha \rangle < 0}}^N \frac{1}{n^p |\langle n \alpha \rangle|^p} = \sum_{\substack{k=1 \\ k \textrm{ even}}}^{K_N} R_p (u_k) + \xi_{2,N}(\alpha) \]
		with some error terms $\xi_{j,N}(\alpha)$, $j=1,2$ that satisfy
		\[ |\xi_{j,N}(\alpha)| \ll \sum_{|k-K_N| \le |K_N^* - K_N| +2} a_k^p . \]
		Lemma \ref{KNlemma} shows that outside a set of $\mu$-measure $\ll (\log N)^{-1/2}$, the previous sum has $\ll (\log N \log \log N)^{1/2}$ terms. An application of Lemma \ref{sumakptailslemma} with $K \approx (\log N \log \log N)^{1/2}$ and $t =(\log N / \log \log N)^{p^2/(2(p+1))}$ thus yields
		\[ \mu \left( \frac{|\xi_{j,N}(\alpha)|}{(\log N)^p} \ge \left( \frac{\log \log N}{\log N} \right)^{\frac{p}{2(p+1)}} \right) \ll \left( \frac{\log \log N}{\log N} \right)^{\frac{p}{2(p+1)}} . \]
		
		Lemma \ref{conditionallemma} remains true with $R$ replaced by $R_p$, using the fact that in the range $1<p<3/2$ the random variables $a_k^{p-1}$ have finite second moment. Repeating the computations in Lemma \ref{expectedvaluelemma} with $R$ replaced by $R_p$ leads to $\mathbb{E}_{\mu} (R_p(u_k) - R_p(a_k)) = \kappa_p + O(e^{-ak})$ with the constant
		\[ \begin{split} \kappa_p &= - \frac{\pi^2}{12 \log 2} \cdot \frac{2^p}{p-1} + \frac{1}{\log 2} \sum_{m=1}^{\infty} \frac{1}{m^{2p}} \sum_{n=1}^{2m^2} n^p \log \left( 1+\frac{1}{n (n+2)} \right) \\ &\phantom{={}} + \frac{\zeta (2p)}{\log 2} \lim_{N \to \infty} \left( \frac{N^{p-1}}{p-1} - \sum_{n=1}^N n^p \log \left( 1+\frac{1}{n(n+2)} \right) \right) . \end{split} \]
		To evaluate the limit in the last line, let us apply the Euler--Maclaurin summation formula
		\[ \sum_{n=1}^N h(n) = \int_1^N h(x) \, \mathrm{d}x + \frac{h(1)+h(N)}{2} + \int_1^N h'(x) (\{ x \} - 1/2) \, \mathrm{d}x \]
		to the function $h(x)=x^p \log (1+\frac{1}{x(x+2)})$, whose derivative is
		\[ h'(x) = x^{p-1} \left( p \log \left( 1+\frac{1}{x(x+1)} \right) - \frac{2}{(x+1)(x+2)} \right) \ll x^{p-3} . \]
		We have
		\[ \int_1^N x^p \frac{1}{x^2} \, \mathrm{d}x = \frac{N^{p-1}}{p-1} - \frac{1}{p-1}, \]
		and using integration by parts,
		\[ \begin{split} \int_1^N x^p &\left( \log \left( 1+\frac{1}{x(x+2)} \right) - \frac{1}{x^2} \right) \, \mathrm{d}x \\ &= - \frac{1}{p+1} \left( \log \frac{4}{3} -1 \right) +O(N^{p-2}) - \int_1^N \frac{x^{p+1}}{p+1} \left( \frac{-2}{x(x+1)(x+2)} + \frac{2}{x^3} \right) \, \mathrm{d}x \\ &=- \frac{1}{p+1} \left( \log \frac{4}{3} -1 \right) - \frac{1}{p+1} \int_1^{\infty} \frac{x^{p-2}(6x+4)}{(x+1)(x+2)} \, \mathrm{d}x +O(N^{p-2}) . \end{split} \]
		Adding the previous two formulas yields the value of $\int_1^N h(x) \, \mathrm{d}x$ up to $o(1)$ error, and finally letting $N \to \infty$ leads to the value of the limit
		\[ \begin{split} \lim_{N \to \infty} \bigg( \frac{N^{p-1}}{p-1} - &\sum_{n=1}^N n^p \log \left( 1+\frac{1}{n(n+2)} \right) \bigg) = \\ & \frac{2}{(p-1)(p+1)} - \frac{p-1}{2(p+1)} \log \frac{4}{3} + \frac{1}{p+1} \int_1^{\infty} \frac{x^{p-2} (6x+4)}{(x+1)(x+2)} \, \mathrm{d}x \\ &-\int_1^{\infty} \left( \{ x \} - \frac{1}{2} \right) x^{p-1} \left( p \log \left( 1+\frac{1}{x(x+2)} \right) - \frac{2}{(x+1)(x+2)} \right) \, \mathrm{d}x . \end{split} \]
		Lemma \ref{uktoaklemma} now reads
		\[ \sum_{\substack{k=1 \\ k \textrm{ odd}}}^{K_N} R_p(u_k) = \sum_{\substack{k=1 \\ k \textrm{ odd}}}^{K_N} R_p(a_k) + \frac{\kappa_p}{2} K_N + \xi_{1,N}(\alpha) \quad \textrm{and} \quad \sum_{\substack{k=1 \\ k \textrm{ even}}}^{K_N} R_p(u_k) = \sum_{\substack{k=1 \\ k \textrm{ even}}}^{K_N} R_p(a_k) + \frac{\kappa_p}{2} K_N + \xi_{2,N}(\alpha) \]
		with some error terms $\xi_{j,N}(\alpha)$, $j=1,2$ that satisfy
		\[ \mu \left( \frac{|\xi_{j,N}(\alpha)|}{(\log N)^p} \ge \frac{1}{(\log N)^{\frac{2p-1}{3}}} \right) \ll \frac{1}{(\log N)^{\frac{2p-1}{3}}} . \]
		Note that $\frac{2p-1}{3}>\frac{p}{2(p+1)}$, thus this error is negligible. Following the steps in Lemma \ref{gausstransformationlemma}, relying on Lemma \ref{p>1fourierlemma} instead of \eqref{hwitha1} yields
		\[ \begin{split} &\mathbb{E}_{\mu_{\mathrm{Gauss}}} (e^{i (t_1 R_p (a_1) + t_2 R_p(a_2))} -1) = \\ &- \frac{\zeta(2p)^{1/p}}{\log 2} \cos \left( \frac{\pi}{2p} \right) \Gamma \left( 1-\frac{1}{p} \right) |t_1|^{1/p} \left( 1-i \, \mathrm{sgn} (t_1) \tan \frac{\pi}{2p} \right) - i \kappa_p' t_1 \\ &- \frac{\zeta(2p)^{1/p}}{\log 2} \cos \left( \frac{\pi}{2p} \right) \Gamma \left( 1-\frac{1}{p} \right) |t_2|^{1/p} \left( 1-i \, \mathrm{sgn} (t_2) \tan \frac{\pi}{2p} \right) - i \kappa_p' t_2 + O \left( |t_1|^{1+\frac{1}{2p}} + |t_2|^{1+\frac{1}{2p}} \right) \end{split} \]
		with the constant
		\[ \kappa_p' = \zeta (2p) d_p + \frac{1}{\log 2} \sum_{m=1}^{\infty} \frac{1}{m^{2p}} \sum_{n=1}^{2m^2} n^p \log \left( 1 + \frac{1}{n(n+2)} \right) , \]
		where $d_p$ is as in \eqref{dpdefinition}. The analogue of Lemma \ref{sumRaklemma} is the following. For any even integer $K \ge 2$, let
		\[ X_K = \frac{1}{\sigma' K^p} \Bigg( \sum_{\substack{k=1 \\ k \textrm{ odd}}}^K R_p(a_k) + \frac{\kappa_p'}{2} K \Bigg) \quad \textrm{and} \quad Y_K = \frac{1}{\sigma' K^p} \Bigg( \sum_{\substack{k=1 \\ k \textrm{ even}}}^K R_p (a_k) + \frac{\kappa_p'}{2} K \Bigg) \]
		with the constant $\sigma' = \zeta(2p) \left( \frac{1}{2\log 2} \cos \left( \frac{\pi}{2p} \right) \Gamma \left( 1-\frac{1}{p} \right) \right)^p$. Then $(X_K, Y_K) \overset{d}{\to} \mathrm{Stab}(1/p,1) \otimes \mathrm{Stab}(1/p,1)$ with rate $K^{-p+1/2} \log K$ in the Kolmogorov metric, and
		\[ \frac{X_K + Y_K}{2^p} \overset{d}{\to} \mathrm{Stab}(1/p,1) \quad \textrm{and} \quad \frac{X_K - Y_K}{2^p} \overset{d}{\to} \mathrm{Stab}(1/p,0) \]
		with rate $\ll K^{-p+1/2}$ in the Kolmogorov metric. Note that $p-\frac{1}{2} > \frac{p}{2(p+1)}$, thus these errors are negligible. Altogether we obtain
		\[ \frac{1}{\sigma_N} \Bigg( \sum_{\substack{n=1 \\ 0< \langle n \alpha \rangle < 1/(2n)}}^N \frac{1}{n^p \langle n \alpha \rangle^p} + E_N, \sum_{\substack{n=1 \\ -1/(2n)< \langle n \alpha \rangle < 0}}^N \frac{1}{n^p |\langle n \alpha \rangle|^p} + E_N \Bigg) \overset{d}{\to} \mathrm{Stab}(1/p,1) \otimes \mathrm{Stab}(1/p,1) \]
		with rate $\ll (\log N)^{-\frac{p}{2(p+1)}} (\log \log N)^{\frac{p}{2(p+1)}}$ in the Kolmogorov metric, where
		\[ \sigma_N = \sigma' K_N^p = \zeta(2p) \left( \frac{6}{\pi^2} \cos \left( \frac{\pi}{2p} \right) \Gamma \left( 1-\frac{1}{p} \right) \right)^p (\log N)^p \left( 1+O \left( \frac{1}{\log N} \right) \right)   \]
		and
		\[ E_N = \frac{\kappa_p' - \kappa_p}{2} K_N = \frac{6 \log 2}{\pi^2} (\kappa_p' - \kappa_p) \log N + O(1) . \]
		Formula \eqref{pawayfromsingularity} finally yields the desired limit law with centering term $c_p \log N$, where $c_p = \frac{6 \log 2}{\pi^2} (\kappa_p' - \kappa_p) - \frac{2^{p-1}}{p-1}$ is the constant in \eqref{cp}. This finishes the proof in the case $1<p<3/2$.
		
		Assume now that $p \ge 3/2$. In this case we can simply use the evaluations \cite{BO1}
		\begin{equation}\label{formulaakp}
			\begin{split} \sum_{\substack{1 \le n < q_K \\ \langle n \alpha \rangle >0}} \frac{1}{n^p \langle n \alpha \rangle^p} &= \zeta (2p) \sum_{\substack{k=1 \\ k \textrm{ odd}}}^K a_k^p + O \left( \sum_{k=1}^K a_k^{p-1} \right), \\ \sum_{\substack{1 \le n < q_K \\ \langle n \alpha \rangle <0}} \frac{1}{n^p |\langle n \alpha \rangle|^p} &= \zeta (2p) \sum_{\substack{k=1 \\ k \textrm{ even}}}^K a_k^p + O \left( \sum_{k=1}^K a_k^{p-1} \right) . \end{split}
		\end{equation}
		Note that at the end of the proof of \cite[Theorem 3]{BO1} only the symmetric case $\sum_{1 \le n < q_K} \frac{1}{n^p \| n \alpha \|^p} = \zeta (2p) \sum_{k=1}^K a_k^p + O \left( \sum_{k=1}^K a_k^{p-1} \right)$ was stated, but the proof actually gives \eqref{formulaakp}. In particular,
		\[ \sum_{\substack{n=1 \\ \langle n \alpha \rangle >0}}^N \frac{1}{n^p \langle n \alpha \rangle^p} = \zeta (2p) \sum_{\substack{k=1 \\ k \textrm{ odd}}}^{K_N} a_k^p + \xi_{1,N}(\alpha) \quad \textrm{and} \quad \sum_{\substack{n=1 \\ \langle n \alpha \rangle <0}}^N \frac{1}{n^p |\langle n \alpha \rangle|^p} = \zeta (2p) \sum_{\substack{k=1 \\ k \textrm{ even}}}^{K_N} a_k^p + \xi_{2,N} (\alpha) \]
		with some error terms $\xi_{j,N} (\alpha)$, $j=1,2$ that satisfy
		\[ |\xi_{j,N}(\alpha)| \ll \sum_{|k-K_N| \le |K_N^* - K_N|+2} a_k^p + \sum_{k=1}^{K_N^*+1} a_k^{p-1} . \]
		Lemma \ref{KNlemma} shows that outside a set of $\mu$-measure $\ll (\log N)^{-1/2}$ the first sum in the previous formula has $\ll (\log N \log \log N)^{1/2}$ terms, whereas the second sum has at most $K_N + O((\log N \log \log N)^{1/2}) \le \log N$ terms. Hence an application of Lemma \ref{sumakptailslemma} gives
		\[ \mu \left( \frac{1}{(\log N)^p} \sum_{|k-K_N| \le |K_N^* - K_N|+2} a_k^p \ge \left( \frac{\log N}{\log \log N} \right)^{-\frac{p}{2(p+1)}} \right) \ll \left( \frac{\log N}{\log \log N} \right)^{-\frac{p}{2(p+1)}} . \]
		In the range $p>(1+\sqrt{5})/2$ (when $p-1/p>1$), Lemma \ref{sumakptailslemma} shows that
		\[ \mu \left( \frac{1}{(\log N)^p} \sum_{k=1}^{K_N^* +1} a_k^{p-1} \ge (\log N)^{-1/p} \right) \ll (\log N)^{-1/p} + (\log N)^{-1/2} . \]
		In the remaining range $3/2 \le p \le (1+\sqrt{5})/2$ Lemma \ref{sumakptailslemma} yields a similar estimate with an error term negligible compared to $(\log N / \log \log N)^{-\frac{p}{2(p+1)}}$. In particular, $\mu (|\xi_{j,N}(\alpha)|/ (\log N)^p \ge \varepsilon_N) \ll \varepsilon_N$ with $\varepsilon_N = (\log N)^{- \frac{p}{2(p+1)}} (\log \log N)^{\frac{p}{2(p+1)}} + (\log N)^{-1/p}$. The desired limit law with rate $\ll \varepsilon_N$ now follows from Lemma \ref{sumakplimitlawlemma}. This finishes the proof in the case $p \ge 3/2$.
	\end{proof}
	
	The proof of Corollary \ref{p>1corollary} is identical.
	
	\section*{Acknowledgments} Bence Borda is supported by the Austrian Science Fund (FWF) project M 3260-N. Lorenz Fr\"uhwirth and Manuel Hauke are supported by the Austrian Science Fund (FWF) project P 35322. We would like to thank the referee for valuable comments.

\end{document}